\documentclass[12pt]{amsart}

\usepackage[latin1]{inputenc}
\usepackage{amssymb,amsmath,amsfonts, amsthm}
\usepackage{enumerate}
\usepackage{times}

\usepackage{bbding}
\usepackage[margin=2.0cm]{geometry}
\usepackage{cite}

\newtheorem{theorem}{Theorem}[section]
\newtheorem{lemma}[theorem]{Lemma}
\newtheorem{proposition}[theorem]{Proposition}

\theoremstyle{definition}
\newtheorem{definition}[theorem]{Definition}
\newtheorem{example}[theorem]{Example}

\theoremstyle{remark}
\newtheorem*{remark}{Remark}

\newcommand{\argmax}{\operatorname{argmax}}
\newcommand{\Hess}{\operatorname{D}^2\!}

\newcommand{\rn}{\R^n}
\newcommand{\K}{{\mathcal K}}
    \newcommand{\sn}{{\mathbb{S}^{n-1}}}
	\newcommand{\Bn}{B^n}

\DeclareMathOperator{\SL}{SL}

\renewcommand{\O}{\operatorname{O}}

\DeclareMathOperator{\oZ}{\operatorname{Z}}
\DeclareMathOperator{\oX}{\operatorname{X}}
\DeclareMathOperator{\oY}{\operatorname{Y}}

\newcommand{\G}{{\mathcal G}}
\newcommand{\R}{{\mathbb R}}
\newcommand{\N}{{\mathbb N}}

\newcommand{\sfe}{{\mathbb S}^{n-1}}
\newcommand{\ml}{{\mathcal H}^{n}}

\newcommand{\diam}{{\rm diam}}
\newcommand{\interno}{\operatorname{int}}
\newcommand{\dist}{\operatorname{dist}}
\newcommand{\dom}{\operatorname{dom}}
\newcommand{\dimension}{\operatorname{dim}}
\newcommand{\relint}{\operatorname{relint}}
\newcommand{\nor}{{\rm Nor}}
\newcommand{\hm}{{\mathcal H}}
\newcommand{\epi}{\operatorname{epi}}
\newcommand{\Borel}{{\mathcal B}}
\newcommand{\bd}{\operatorname{bd}}
\renewcommand{\d}{\,\mathrm{d}}

\newcommand{\reg}[2]{\operatorname{reg_{\it #2}\!}{#1}}

\newcommand{\fconv}{{\mbox{\rm Conv}_{\text{coe}}(\R^n)}} 
\newcommand{\fconvone}{{\mbox{\rm Conv}_{\text{coe}}(\R)}} 
\newcommand{\fconvx}{{\mbox{\rm Conv}(\R^n)}}
\newcommand{\fconvh}{{\mbox{\rm Conv}_{\text{\rm hom}}(\R^n)}} 
\newcommand{\fconvxk}{{\mbox{\rm Conv}(\R^k)}} 
\newcommand{\faff}{{\mbox{\rm Aff}(\R^n)}}

\numberwithin{equation}{section}

\begin{document}

\title{Hessian valuations}



\address{Andrea Colesanti:\
Dipartimento di Matematica e Informatica ``U. Dini''
Universit\`a degli Studi di Firenze,
Viale Morgagni 67/A - 50134, Firenze, Italy}
\email{colesant@math.unifi.it}

\address{Monika Ludwig:\
Institut f\"ur Diskrete Mathematik und Geometrie,
Technische Universit\"at Wien,
Wiedner Hauptstra\ss e 8-10/1046,
1040 Wien, Austria}
\email{monika.ludwig@tuwien.ac.at}

\address{Fabian Mussnig:\
Institut f\"ur Diskrete Mathematik und Geometrie,
Technische Universit\"at Wien,
Wiedner Hauptstra\ss e 8-10/1046,
1040 Wien, Austria}
\email{fabian.mussnig@alumni.tuwien.ac.at}

\author{Andrea Colesanti, Monika Ludwig \& Fabian Mussnig}


\date{}


\begin{abstract} A new class of continuous valuations on the space of convex functions on $\R^n$ is introduced. On smooth convex functions, they are defined for $i=0,\dots,n$ by 
\begin{equation*}
u\mapsto \int_{\R^n} \zeta(u(x),x,\nabla u(x))\,[\Hess u(x)]_i\d x
\end{equation*}
where $\zeta\in C(\R\times\R^n\times\R^n)$ and $[\Hess u]_i$ is the $i$th elementary symmetric function of the eigenvalues of the Hessian matrix, $\Hess u$,  of $u$. Under suitable assumptions on $\zeta$, these valuations are shown to be invariant under translations and rotations on convex and coercive functions.

\bigskip


\end{abstract}

\maketitle

\section{Introduction}

The purpose of this paper is to introduce a new type of valuations on convex functions. On smooth convex functions, they are defined for $i=0,\dots,n$ by  
\begin{equation}\label{intro 1}
u\mapsto \int_{\R^n} \zeta(u(x),x,\nabla u(x))\,[\Hess u(x)]_i \d x,
\end{equation}
where $\zeta: \R\times\R^n\times \R^n \to \R$ is continuous and $[\Hess u]_i$ is the $i$th elementary 
symmetric function of the eigenvalues of the Hessian matrix of $u$, with the usual convention 
$[\Hess u]_0=1$ (alternatively, $[\Hess u]_i$ is the sum of the
$i\times i$ principal minors of  $\Hess u$). We show that these functionals can be extended to a rather ample class of 
convex functions, providing a family of continuous valuations that we call \emph{Hessian valuations}. Here continuity is with respect to epi-convergence (see Section \ref{Preliminaries}). Under suitable assumptions on $\zeta$, Hessian valuations are invariant under rotations on convex functions and under translations and rotations on convex and coercive functions.

\medskip

The theory of valuations on function spaces has been rapidly growing in 
recent years. Several spaces of functions have been investigated already, including Lebesgue and Sobolev spaces, functions of bounded variations, and
quasi-concave functions (see 
\cite{Alesker-17,BaryshnikovGhristWright, BobkovColesantiFragala, ColesantiFragala, ColesantiLombardi, Colesanti-Lombardi-Parapatits,Kone,Ludwig:Fisher,Ludwig:survey,Ludwig:sobval,Ludwig:MM,Ma2016,Ober2014,Tsang:Lp,Tsang:Minkowski,Tuo_Wang_semi}). 
Needless to say, the major impulse to this area comes from the rich and beautiful theory of valuations defined on the family 
$\K^n$ of convex bodies (that is, compact convex subsets) in $\R^n$, which is one of the most active branches of convex geometry (see \cite[Chapter 6]{Schneider} for a recent survey on the theory of valuations on convex bodies).

\medskip

As a general framework, we start from the following family of convex functions:
$$
\fconvx=\{u\colon\R^n\to\R\cup\{+\infty\}\colon u\, \mbox{ convex and l.s.c.},\ u\not\equiv+\infty\},
$$
where l.s.c.\ stands for lower semicontinuous.  A (real-valued) valuation on $\fconvx$ is a functional $\oZ\colon\fconvx\to\R$ which has the additivity property
$$
\oZ(u\vee v)+\oZ(u\wedge v)=\oZ(u)+\oZ(v),
$$
for every $u,v\in\fconvx$ such that $u\wedge v\in\fconvx$. Here $\vee$ and $\wedge$ denote the pointwise maximum and minimum, respectively. 

\medskip\goodbreak

The basic tool to extend a functional of the form \eqref{intro 1} to $\fconvx$ are {\em Hessian measures} of convex functions (see Section \ref{Hessian measures}), that will be denoted by $\Theta_i(u,\cdot)$ for  $i\in\{0,\dots, n\}$ and $u\in\fconvx$. 
Like support measures of convex bodies, Hessian measures can be defined as coefficients of a local Steiner formula,
\begin{equation}\label{Steiner intro}
\hm^n(P_s(u,\eta))=\sum_{i=0}^n\binom nis^i\,\Theta_{n-i}(u,\eta),
\end{equation}
where $\hm^n$ is the $n$-dimensional Hausdorff measure and $s\ge0$ while for a Borel subset $\eta$  of $\R^{n}\times\R^n$, 
$$
P_s(u,\eta)=\{x+s y\colon (x,y)\in\eta,\ y\in\partial u(x)\}.
$$
Here $\partial u(x)$ denotes the subdifferential of $u$ at the point $x$  (see Section \ref{Preliminaries}).
Hessian measures permit to extend to non-smooth convex functions integrals of the elementary 
symmetric functions of the eigenvalues of the Hessian matrix. 
Indeed, if $u\in\fconvx\cap C^2(\R^n)$ and $i\in\{0,\dots, n\}$, then
$$
\Theta_i(u,\beta\times\R^n)=\int_{\beta}\,[\Hess u(x)]_{n-i}\d x
$$
for every Borel subset $\beta$ of $\R^n$. 
Hessian measures have been considered in connection with convex (and more general) functions in 
\cite{Colesanti-1997,Colesanti-Hug-2000,Colesanti-Hug-2000b, Fu-1989} and are related to non-linear elliptic partial differential equations (see, {\em e.g.}, \cite{Caffarelli-Nirenberg-Spruck, Colesanti-Salani, Trudinger:Wang1999}). 

\medskip

To define the integral \eqref{intro 1} for an arbitrary $u\in\fconvx$, we integrate $\zeta(u(x),x,y)$ with respect to the Hessian measure of $u$, where the variable $y$ plays the role of $\nabla u$. 
To guarantee integrability, we assume that $\zeta(t,x,y)$ has compact support with respect to the second and third variables, that is, there exists $r>0$ such that $\zeta$ vanishes in the complement set of the cylinder $\{(t,x,y)\colon|x|\le r, |y|\le r\}$, where $\vert \cdot\vert$ is the Euclidean norm. Our first result is the following.

\begin{theorem}\label{theorem intro 1} Let $\zeta\in C(\R\times\R^n\times\R^n)$ have compact support with respect to 
the second and third variables. For every $i\in\{0,1,\dots,n\}$, the functional $\oZ_{\zeta,i}\colon {\fconvx}\to\R$, defined by
\begin{equation}\label{Hessian valuations intro}
\oZ_{\zeta,i}(u)=\int_{\R^{2n}}\zeta(u(x),x,y)\d\Theta_i(u,(x,y)),
\end{equation}
is a continuous valuation on $\fconvx$. If $u\in\fconvx\cap C^2(\R^n)$, then $\oZ_{\zeta,i}(u)$ takes the form \eqref{intro 1}. 
\end{theorem}

\smallskip\noindent
For $\zeta$ and $i$ given as above, two additional properties of $\oZ_{\zeta,i}$ will be established.
\begin{enumerate}[(i)]
\setlength\itemsep{0.5em}
\item The functional $\oZ_{\zeta,i}$ is $i$-simple, that is, we have $\oZ_{\zeta,i}(u)=0$ for all  $u\in\fconvx$ such that $\dimension(\dom(u))<i$,
where $\dom(u)=\{x\in\R^n\colon u(x)<+\infty\}$ is the domain of $u$ and dim stands for dimension.
\item If $\zeta$ is of the form $\zeta(t,x,y)=\xi(t,|x|,|y|)$ for some $\xi\in C(\R\times[0,+\infty)^2)$,
then \eqref{Hessian valuations intro} is invariant under rotations, 
that is, $\oZ_{\zeta,i}(u)=\oZ_{\zeta,i}(u\circ \phi^{-1})$ for every $\phi\in\O(n)$.
\end{enumerate}

\begin{remark}\label{remark Alesker} 
In \cite{Alesker-17}, Alesker considers the space of convex functions on an open and convex set $U\subset \R^n$. Using Monge-Amp\`ere measures, he introduces a class of valuations which extend to general convex functions functionals of the form
$$
u\mapsto\int_U \xi(x)\det(\Hess u(x), \dots,\Hess u(x) ,A_{k+1}(x),\dots,A_n(x))\d x,
$$
where  the function $\xi\in C(\R^n)$ has compact support,  $\det$ is the mixed discriminant operator of $n$ matrices, the Hessian $\Hess u(x)$ appears $k$ times with $k\in\{0,\dots,n\}$, and $A_i$ is a symmetric $n\times n$ matrix having as coefficients continuous 
and compactly supported functions for every $i\in\{k+1,\dots,n\}$. There is an overlap between Alesker's valuations and the ones 
introduced here. Indeed, for $U=\R^n$, if $\zeta=\xi$ depends on the $x$-variable only and we choose $A_i\equiv I_n$ for $i=k+1,\dots,n$, where $I_n$ is
the identity matrix of order $n$, then the notions coincide for every $u\in\fconvx\cap C^2(\R^n)$.
\end{remark}

\medskip\goodbreak

In \cite{CavallinaColesanti, Colesanti-Ludwig-Mussnig-1, Colesanti-Ludwig-Mussnig-2, Mussnig}, valuations  defined
on the space of convex and coercive functions,
$$
\fconv=\{u\colon\R^n\to\R\cup\{+\infty\}\colon u\, \mbox{ convex, l.s.c., and coercive},\ u\not\equiv+\infty\}
$$
were studied and classified. Here a function $u:\R^n\to \R\cup\{+\infty\}$ is coercive if  $\lim_{|x|\to+\infty}u(x)=+\infty$.
In \cite{CavallinaColesanti, Colesanti-Ludwig-Mussnig-1}, the following type 
of  valuations on $\fconv$ were considered  for suitable $\zeta:\R\to \R$:
\begin{equation}\label{intro 3}
u\mapsto \int_{\R^n} \zeta(u(x))\d x.
\end{equation} 
In \cite{Colesanti-Ludwig-Mussnig-1}, it is proved that \eqref{intro 3} defines a continuous and rigid motion invariant valuation on 
$\fconv$ if and only if $\zeta$ is continuous on $\R$ and has finite $(n-1)$th moment. Moreover, it is proved that every continuous, 
non-negative, $\SL(n)$ and translation invariant 
valuation on $\fconv$ can be written as a functional of the form (\ref{intro 3}) plus a function only depending on $\min_{\R^n}u$. 

\medskip

We obtain new valuations on $\fconv$. The coercivity
guarantees that integrability is preserved under less restrictive conditions on $\zeta$. In particular, 
we may remove the dependence on the space variable $x$, gaining translation invariance. 

\begin{theorem}\label{theorem intro 2}
Let $\zeta\in C(\R\times \R^n)$ have compact support. For every $i\in\{0,\dots, n\}$, the functional 
$\oZ_{\zeta,i}\colon\fconv\to\R$, defined by
\begin{equation}\label{Hessian valuations 2 intro}
\oZ_{\zeta,i}(u)=\int_{\R^{2n}}\zeta(u(x),y)\d\Theta_i(u,(x,y)),
\end{equation}
is a continuous, translation invariant, $i$-simple valuation. Moreover, if there exists $\xi\in C(\R\times[0,+\infty))$ such that $\zeta(t,y)=\xi(t,|y|)$ on $\R\times\R^n$, then $\oZ_{\zeta,i}$ is also rotation invariant.
\end{theorem}

Hence, for every $\zeta\in C(\R\times[0,+\infty))$ with compact support, the valuation
$$
u\mapsto \int_{\R^{2n}}\zeta(u(x),|y|)\d\Theta_i(u,(x,y))
$$
is continuous and rigid motion invariant for  $i\in\{0,\dots, n\}$. 

\medskip

We mention that in \cite{CavallinaColesanti} the construction of a class of valuations on $\fconv$ is presented, based on 
quermassintegrals (or intrinsic volumes) of level sets. Similar ideas can be found in 
\cite{BobkovColesantiFragala,Milman-Rotem,Milman-Rotem2}. A
comparison between those and Hessian valuations is carried out in Section \ref{sub-section comparison}, where in particular
we show that there are Hessian valuations that are not included in the class described in \cite{CavallinaColesanti}.

\medskip

In the following, we collect results needed for the preparation of the proofs of the main results.  In Section~\ref{Convex functions}, we recall some basic facts on convex functions.  In Section \ref{Approximation},  the Lipschitz regu\-larization of convex functions is described.
Sections \ref{Inclusion-exclusion properties of the subgradient map} and \ref{The graph of the subgradient map and the parallel set of a function} are devoted
to inclusion-exclusion type properties of the subdifferential, the subdifferential graph and local parallel sets of convex functions. These properties will
be critical to prove the valuation property of Hessian measures, which is discussed in Sections \ref{Hessian measures} and
\ref{Further properties of Hessian measures}.
In particular, in Section \ref{Hessian measures} we extend to $\fconvx$ the results of \cite{Colesanti-Hug-2000} concerning existence and integral representations of Hessian
measures, while in the next section we discuss several (known and new) facts about them.

Subsequently, after giving the main definitions for valuations on $\fconvx$ in Section \ref{Valuations}, we prove Theorems \ref{theorem intro 1} and \ref{theorem intro 2}
in Sections \ref{Hessian valuations} and \ref{Valuations on fconv}.  
In Section \ref{The space of 1-homogeneous convex functions}, we briefly analyze what happens if we restrict Hessian valuations to the subset $\fconvh$
of $\fconvx$, formed  by finite convex functions homogeneous of degree 1, that is, by support functions of convex bodies. Note that continuous valuations on 
$\fconvh$ correspond to continuous valuations on $\K^n$.
 
\section{Preliminaries}\label{Preliminaries} 

We work in $n$-dimensional Euclidean space, $\R^n$, for $n\ge1$, endowed with the usual Euclidean norm $|\cdot|$ and 
scalar product $\langle\cdot,\cdot\rangle$. We set $\sn=\{x\in\R^n\colon |x|=1\}$ and $\Bn=\{x\in\R^n: \vert x \vert \le 1\}$. For $r>0$, let $\Bn_r$ be the closed ball of $\R^{n}$, centered at the origin and with radius $r$. 
Given a subset $A$ of $\R^n$, its interior and  boundary will be denoted by $\interno(A)$ and $\bd(A)$, respectively.
For $k\in[0,n]$, the $k$-dimensional Hausdorff measure in $\R^n$ is denoted by $\mathcal{H}^k$. In particular, $\ml$ is the Lebesgue measure 
in $\R^n$. 
We write $\pi_1$ and $\pi_2$  for the canonical projections of $\R^n\times\R^n$ onto the first and the second component, respectively; that is,
$$
\pi_1(x,y)=x,\quad \pi_2(x,y)=y
$$
for $(x,y)\in\R^n\times\R^n$.
\medskip

A subset $C$ of $\R^n$ is convex if for every $x_0, x_1\in C$ and $t\in[0,1]$, we have $(1-t)x_0+tx_1\in C$. If $C\subset\R^n$ is convex, we define
its dimension, $\dimension(C)$, as the minimum integer $k\in\{0,1,\dots,n\}$ such that there exists an affine subspace  of dimension $k$ containing $C$. 
The relative interior of a convex set $C$ of dimension $k$ is the subset of those points
$x$ of $C$ for which there exists a $k$-dimensional ball centered at $x$ and contained in $C$. The relative 
interior will be denoted by $\relint(C)$. Note that $\dimension(C)=n$ if and only if $\relint(C)=\interno(C)$. The group of rotations of $\R^n$ will be denoted by $\O(n)$. By a {\em rigid motion} we mean the composition of a translation and an element of
$\O(n)$.

\medskip

A {\em convex body} is a compact convex subset of $\R^n$. We will denote by $\K^n$ the family of convex bodies in $\R^n$.
Our main reference for the theory of convex body is the monograph \cite{Schneider}.
We will sometimes need separation theorems, especially for convex sets. For the notion of separation, strict separation and strong separation
of two subsets of $\R^n$ by a hyperplane we refer to\cite[Section 1.3]{Schneider}.

\medskip
If $u$ and $v$ are functions defined in $\R^n$, taking values in $\R$ or in $\R\cup\{+\infty\}$, we denote by $u\vee v$ and $u\wedge v$ the 
pointwise maximum function and the pointwise minimum function of $u$ and $v$, respectively. In other words, for $x\in\R^n$, 
$$
(u\vee v)(x)=\max\{u(x),v(x)\},\quad
(u\wedge v)(x)=\min\{u(x),v(x)\}.
$$
For $u\colon\R^n\to\R\cup\{+\infty\}$, we define its epigraph by
$$
\epi(u)=\{(x,t)\in\R^n\times\R\colon t\ge u(x)\}.
$$
We say that a function $w\colon\R^n\to\R$ is affine if there exist $y\in\R^n$ and $c\in\R$ such that
$
w(x)=\langle x,y\rangle +c$
for $x\in\R^n$.  We denote by $\faff$ the family of affine functions on $\R^n$.

\section{Convex Functions}\label{Convex functions}

A function $u\colon\R^n\to\R\cup\{+\infty\}$ is convex if for every $x_0, x_1\in\R^n$ and for every $t\in[0,1]$, we have
$$
u((1-t)x_0+tx_1)\le(1-t)u(x_0)+t u(x_1).
$$
Recall that a function $u$ is convex  if and only if  
the {epigraph} of $u$ is a convex subset of $\R^n\times\R$. Note that  the lower semicontinuity of $u$ is equivalent to $\epi(u)$ being closed. Such functions are also called \emph{closed}. 
Our main reference texts on convex analysis -- the theory of convex functions -- are \cite{Rockafellar} and
\cite{Rockafellar-Wets}.

\medskip

For $u\in\fconvx$, we define the  {\it domain} of $u$ as
$$
\dom(u):=\{x\in\R^n\colon  u(x)<+\infty\}.
$$
By the convexity of $u$, its domain is a convex set. Every convex function $u$ is continuous in the interior of $\dom(u)$.
\goodbreak
\medskip

\begin{example}\label{example 1}
{\rm Let $K$ be a convex body in $\R^n$. The {\it convex indicator function},
$I_K\colon \R^n\to\R\cup\{+\infty\}$,  of $K$ is defined by
$$
I_K(x)=
\left\{
\begin{array}{ll}
0\quad&\mbox{if $x\in K$,}\\
+\infty\quad&\mbox{if $x\notin K$.}
\end{array}
\right.
$$
Hence $I_K\in\fconv$ for every $K\in\K^n$.}
\end{example}

\begin{example}\label{example 2}
{\rm Another function  associated to a convex body $K$ is its {\it support function}, denoted by $h_K$ and defined on $\R^n$ by
$$
h_K(y)=\sup_{x\in K}\langle x,y\rangle
$$
(see \cite[Chapter 1]{Schneider}). The support function $h_K$ is a 1-homogeneous convex function for every $K\in\K^n$, where $f:\R^n\to \R$ is called 1-homogeneous, if $f(t\,x)=t\,u(x)$ for all $t> 0$ and $x\in \R^n$. In particular,
the support function is everywhere finite by the compactness of $K$.}
\end{example}

The pointwise maximum of two convex functions is again a convex function, but this does not guarantee that $\fconvx$ is closed with respect to
$\vee$, as $u\vee v$ may be identically $+\infty$ for some $u,v\in\fconvx$. On the other hand, it is easy to see that if $u,v\in\fconvx$ are such that
$u\wedge v\in\fconvx$, then $u\vee v\in\fconvx$.

\subsection{The conjugate of a convex function}
We recall the notion of conjugate of a convex function (see \cite{Rockafellar}). For $u\in\fconvx$ and $y\in\R^n$, we set
$$
u^*(y):=\sup\nolimits_{x\in\R^n}\big(\langle x,y\rangle-u(x)\big).
$$
As every function $u\in\fconvx$ is closed with $u\not\equiv+\infty$, the following result is a consequence of  Theorem~12.2 and Corollary 12.2.1 in \cite{Rockafellar}.

\begin{proposition}\label{Schneider proposition} If $u\in\fconvx$, then $u^*\in\fconvx$ and $u^{**}:=(u^*)^*=u$.
\end{proposition}


\goodbreak
The next result follows easily from \cite[Theorem 16.5]{Rockafellar}.

\begin{proposition}\label{minmaxconjugates} If $u, v\in\fconvx$ are such that 
$u\wedge v\in\fconvx$, then 
\begin{equation}\label{relazione 1}
(u\wedge v)^*=u^*\vee v^*,
\end{equation}
and
\begin{equation}\label {relazione 2}
(u\vee v)^*=u^*\wedge v^*.
\end{equation}
In particular $u^*\wedge v^*$ is convex.
\end{proposition}

\begin{proof}
For every $y\in\R^n$, we have
\begin{eqnarray*}
(u\wedge v)^*(y)&=&\sup\nolimits_{x\in\R^n}\big(\langle x,y\rangle-(u\wedge v)(x)\big)\\
&=&\sup\nolimits_{x\in\R^n}\big(\max\big\{\langle x,y\rangle -u(x),\langle x,y\rangle-v(x)\big\}\big)\\
&=&\max\big\{\sup\nolimits_{x\in\R^n}\big(\langle x,y\rangle-u(x),\sup\nolimits_{x\in\R^n}\big(\langle x,y\rangle-v(x)\big)\big\}\\
&=&(u^*\vee v^*)(y).
\end{eqnarray*}
This proves \eqref{relazione 1}.  Together with Proposition \ref{Schneider proposition} this gives
$$(u^* \wedge v^*) = ((u^* \wedge v^*)^*)^* = ((u^*)^* \vee (v^*)^*)^*=(u \vee v)^*,$$
which shows \eqref{relazione 2}.
\end{proof}


\subsection{Topology in $\fconvx$}\label{section topology}

As in \cite{Colesanti-Ludwig-Mussnig-1} and \cite{Colesanti-Ludwig-Mussnig-2}, we adopt the topology induced by {\em epi-convergence}. 
A sequence $u_k$ of elements of $\fconvx$ is epi-convergent to $u\in\fconvx$ if for every $x\in\R^n$ the following conditions hold.
\begin{itemize}
	\item[(i)] For every sequence $x_k$ that converges to $x$,
			\begin{equation*}\label{eq:gc_inf}
				u(x) \leq \liminf_{k\to +\infty} u_k(x_k).
			\end{equation*}
	\item[(ii)] There exists a sequence $x_k$ that converges to $x$ such that
			\begin{equation*}\label{eq:gc_sup}
				u(x) = \lim_{k\to+\infty} u_k(x_k).
			\end{equation*}
\end{itemize}

An exhaustive source for epi-convergence of convex functions is \cite{Rockafellar-Wets}, where also the following result can be found.

\begin{proposition}\label{convergence conjugates} A sequence $u_k$ of functions from $\fconvx$ epi-converges to $u\in\fconvx$
if and only if the sequence $u_k^*$ epi-converges to $u^*$. 
\end{proposition}

\begin{remark} If $u_k$ is a sequence of {\em finite} functions in $\fconvx$, that is,  $\dom(u_k)=\R^n$ for every $k$, and $u\in\fconvx$
is also finite, then $u_k$ epi-converges to $u$ if and only if it converges to $u$ pointwise in $\R^n$ and uniformly on compact
sets (see for instance \cite[Theorem 7.17]{Rockafellar-Wets}).  
\end{remark}

\subsection{Subdifferentials}

Let $u\in\fconvx$ and $x\in\R^n$. A vector $y\in\R^n$ is said  to be a {\em subgradient} of $u$ at  $x$ if
$$
u(z)\ge u(x)+\langle z-x,y\rangle
$$
for all $z\in\R^n$. The (possibly empty) set of all vectors with this property will be denoted by $\partial u(x)$ and called the 
\emph{subdifferential} of $u$ at $x$. In particular, $\partial u(x)=\emptyset$ for every $x\notin\dom(u)$. Also note that
if $x\in\R^n$ is such that $\partial u(x)\ne \emptyset$, then 
$\partial u(x)$ is closed and convex.

\begin{remark}\label{non-empty subgradient}If $u\in\fconvx$, then there exists at least one point $x\in\R^n$ such that $\partial u(x)\ne\emptyset$, since $\dom(u)\ne\emptyset$ implies that  $\relint(\dom(u))\ne\emptyset$ and by \cite[p.\ 227]{Rockafellar}.\end{remark}

The subdifferential is strongly connected to directional derivatives. Let $u\in\fconvx$. Following the notation used in \cite{Rockafellar}, for  $x\in\dom(u)$ and a vector $y\in\R^n$,  we set
$$
u'(x;y)=\lim_{t\to0^+}\frac{u(x+ty)-u(x)}{t}.
$$ 
This limit always exists, finite or not, that is, $u'(x;y)\in\R\cup\{\pm\infty\}$ (see \cite[Theorem 23.1]{Rockafellar}). The following result is  \cite[Theorem 23.2]{Rockafellar}.

\begin{lemma}\label{rock lemma} Let $u\in\fconvx$ and $x\in\dom(u)$. For $y\in\R^n$, we have $y\in\partial u(x)$ if and only if
$\langle z,y\rangle\le\ u'(x;z)$ for all $z\in\R^n$.
\end{lemma}

We also need the following result, which can be found in \cite[Theorem 23.5]{Rockafellar}.

\begin{lemma}
\label{le:conjugate_subgradient}
For $u\in\fconvx$ and $x,y\in\rn$, the following are equivalent:
	\begin{enumerate}[(i)]
		\item $y \in \partial u(x),$
		\item $x \in \partial u^*(y),$
		\item $\langle x,y\rangle = u(x)+ u^*(y),$
		\item $x\in \argmax_{z\in\rn} \left(\langle y,z\rangle-u(z)\right),$
		\item $y\in \argmax_{z\in\rn} \left(\langle x,z\rangle-u^*(z) \right).$
	\end{enumerate}
\end{lemma}

\noindent Here $\argmax\nolimits_{z\in V} v(z)$ denotes the set of points in the set $V$ at which the function values of $v$ are maximized on $V$.

\section{Lipschitz Regularization}\label{Approximation}

For a function $u\in\fconvx$ and $r>0$, we consider its \emph{Pasch-Hausdorff envelope} or \emph{Lipschitz regularization}
$$\reg{u}{r} =(u^* + I_{\Bn_{1/r}})^*$$
(see \cite{HU1,HU2} and  \cite[Example 9.11]{Rockafellar-Wets}). 

\goodbreak

\begin{remark} 
Given $u\in\fconvx$ and $r>0$, the Lipschitz regularization $\reg{u}{r}$ admits the following equivalent definition
$$
\reg{u}{r}(x)=\sup\{w(x)\colon\mbox{$w\in\faff$, $w\le u$ in $\R^n$, $|\nabla w|\le 1/r$}\}.
$$
In other words, the graph of $\reg{u}{r}$ is the envelope of all supporting hyperplanes to the graph of $u$, with slope 
bounded by $1/r$. This can be also rephrased as follows: $\reg{u}{r}$ is the largest convex function which is smaller than $u$
and has Lipschitz constant bounded by $1/r$ (for sufficiently large $r$). For brevity, we omit the proof of the equivalence of these definitions.


\end{remark}

\begin{proposition}\label{approximation lemma} For $u\in\fconvx$ and $r>0$, the Lipschitz regularization has the following
properties.
\begin{enumerate}[(i)]
\item For $r>0$, the function $\reg{u}{r}$ is convex in $\rn$. There exists $r_0>0$ such that $\reg{u}{r}(x)>-\infty$ for every $0<r\le r_0$ and $x\in\rn$. \label{eins}
\item For $r>0$, we have $\reg{u}{r}<+\infty$ on $\R^n$.
\item For $0<r<t$, we have  $\reg{u}{t}(x)\le \reg{u}{r}(x)\le u(x)$  for every $x\in\R^n$.
\item As $r\to 0$, the functions  $\reg{u}{r}$ epi-converge to $u$.
\item For every $x\in\rn$, we have
$\lim_{r\to0} \reg{u}{r}(x)=u(x)$.
\item If $x\in\dom(u)$ and $y\in\partial u(x)$, then $\reg{u}{r}(x)=u(x)$ and $\partial \reg{u}{r}(x)=\partial u(x)\cap \Bn_{1/r}$ for $|y|<\frac1{r}$. \label{sechs}

\end{enumerate}
\end{proposition}

\begin{proof}
(i) Since $\reg{u}{r}$ is defined as the conjugate of a convex function, it is convex itself. Furthermore, since $u^*\not\equiv+\infty$, there exists $r_0>0$ such that $\dom (u^*)\cap \Bn_{1/r} \neq \emptyset$ for every $0<r\le r_0$ and therefore $u^*+I_{\Bn_{1/s}}\in\fconvx$. Hence, Proposition \ref{Schneider proposition} implies that $\reg{u}{r}\in\fconvx$ for every $0<\le r_0$.

\noindent
(ii) It follows from the definition of the convex conjugate, that
$$\reg{u}{r}(x) = \sup_{y\in\rn} \big(\langle x,y\rangle-(u^*(y)+I_{\Bn_{1/r}}(y))\big) = \sup_{y\in \Bn_{1/r}}\big(\langle x,y \rangle -u^*(y)\big) < +\infty.$$

\noindent
(iii) This follows from the fact that convex conjugation is order reversing, that is, if $u\leq v$ pointwise, then $u^* \geq v^*$ pointwise.

\noindent
(iv) By Proposition \ref{convergence conjugates}, it is enough to show that $u^*+I_{\Bn_{1/r}}$ is epi-convergent to $u^*$, which follows from the properties of epi-convergence (see also \cite[Theorem 7.46]{Rockafellar-Wets}).

\noindent
(v) By (iii), the function $\reg{u}{r}(x)$ is bounded above by $u(x)$ and decreasing in $r$. Hence, $\lim_{r\to 0} \reg{u}{r}(x)$ exists in $(-\infty,+\infty]$. Together with (iv) and the definition of epi-convergence, this gives
$$u(x) \leq \lim_{s\to 0} \reg{u}{r}(x) \leq u(x).$$

\noindent
(vi) For $x\in\dom(u)$, we have by Lemma \ref{le:conjugate_subgradient}
\begin{align*}
	y \in \partial u(x),\quad |y|\leq \frac1s \quad
	&\Longleftrightarrow \quad y\in \operatorname{argmax}_{z\in \Bn_{1/r}} \left( \langle x, z\rangle - u^*(z) \right)\\
	&\Longleftrightarrow \quad y\in \operatorname{argmax}_{z\in\rn} \big( \langle x, z\rangle - (u^*(z)+I_{\Bn_{1/r}}(z)) \big)\\
	&\Longleftrightarrow \quad y\in \operatorname{argmax}_{z\in\rn} \left( \langle x, z\rangle - (\reg{u}{r})^* (z) \right)\\
	&\Longleftrightarrow \quad y \in\partial \reg{u}{r}(x).
\end{align*}
Hence, for $y\in\partial u(x)$ with $\vert y\vert <\frac1r$,
$$u(x) = \sup\nolimits_{z\in \Bn_{1/r}} \big( \langle x,z\rangle  - u^*(z) \big) = \sup\nolimits_{z\in\rn} \big(\langle x,z\rangle - u^*(z)- I_{\Bn_{1/r}}(z)\big) = \reg{u}{r}(x),$$
which concludes the proof.
\end{proof}


\begin{proposition}\label{approximation lemma 2} Let $u_j, u\in\fconvx$. If $u_j$ epi-converges to $u$, then
$\reg{u_j}{r}$ epi-converges to $\reg{u}{r}$ as $j\to+\infty$ for sufficiently small $r>0$.
\end{proposition}

\begin{proof} By Proposition \ref{convergence conjugates}, it is enough to prove that given a sequence 
$v_j\in\fconvx$ that is epi-convergent to $v\in\fconvx$, then, for $s>0$ sufficiently large, the sequence $v_j+I_{\Bn_s}$ epi-converges 
to $v+I_{\Bn_s}$. 

Let us fix $s>0$ such that there exists $\bar x\in\dom(v)\cap \Bn_{s/4}$. By epi-convergence, we may assume that for every $j$ there exists $\bar x_j\in \Bn_{s/2}$ 
such that $v_j(\bar x_j)\le 2 v(\bar x)$. Let $x\in\R^n$ and $x_j$ be a sequence converging to $x$. The inequality
$$
\liminf_{j\to+\infty}(v_j+I_{\Bn_s})(x_j)\ge (v+I_{\Bn_s})(x)
$$
follows easily from the epi-convergence of the sequence $v_j$ to $v$. To conclude, we have to show that there exists a specific sequence 
$x_j$ converging to $x$ and such that 
\begin{equation}\label{equality}
\lim_{j\to+\infty}(v_j+I_{\Bn_s})(x_j)=(v+I_{\Bn_s})(x).
\end{equation}
If either $|x|<s$ or $|x|> s$, this is again a straightforward consequence of epi-convergence. So assume that $|x|=s$ and let 
$x_j$ be a sequence such that
$$
\lim_{j\to+\infty}v_j(x_j)=v(x).
$$
If all but a finite number of the elements of $x_j$ belong to $\Bn_s$, then \eqref{equality} follows. Hence we may assume that
$|x_j|>s$ for every $j$. Let $\hat x_j$ be the intersection of the segment joining $x_j$ and $\bar x_j$ with $\bd(\Bn_s)$. Then
$$
\hat x_j=(1-t_j)\bar x_j+t_j x_j
$$
for a suitable $t_j\in[0,1]$. As $|x|=|x_j|=s$ and $|\bar x_j|\le {s/2}$ for every $j$, up to extracting a subsequence, 
we may assume that $t_j\to1$ as $j\to+\infty$, so that, in particular
$$
\lim_{j\to+\infty}\hat x_j=x.
$$
Hence
$$
\liminf_{j\to+\infty}(v_j+I_{\Bn_s})(\hat x_j)
=\liminf_{j\to+\infty} v_j(\hat x_j)\ge v(x)=(v+I_{\Bn_s})(x).
$$
On the other hand, by convexity
$$
(v_j+I_{\Bn_r})(\hat x_j)=v_j(\hat x_j)\le(1-t_j)v_j(\bar x_j)+t_j v_j(x_j)
$$
for $j\in\N$. Passing to the limit as $j\to+\infty$ (and recalling that the sequence $v_j(\bar x_j)$ is bounded from above), we obtain
$$
\limsup_{j\to+\infty}(v_j+I_{\Bn_s})(\hat x_j)\le v(x)=(v+I_{\Bn_s})(x)
$$
which concludes the proof.
\end{proof}

\begin{proposition}\label{lemma 3} Let  $u,v\in\fconvx$. If $u\wedge v\in\fconvx$, then
$$
\reg{(u\wedge v)}{r}=\reg{u}{r}\wedge \reg{v}{r},\quad
\reg{(u\vee v)}{r}=\reg{u}{r}\vee \reg{v}{r}
$$
for sufficiently small $r>0$.
\end{proposition}

\begin{proof} By Proposition \ref{minmaxconjugates}, the function $u^*\wedge v^*$ is convex. Hence,
$$
(\reg{u}{r})^*\wedge (\reg{v}{r})^*=(u^*+I_{\Bn_{1/r}})\wedge(v^*+I_{\Bn_{1/r}})=
(u^*\wedge v^*)+I_{\Bn_{1/r}}
$$
is convex as well, and this implies that $\reg{u}{r}\wedge \reg{v}{r}$ is convex. Consequently
\begin{eqnarray*}
\reg{(u\wedge v)}{r}&=&((u\wedge v)^*+I_{\Bn_{1/r}})^*\\
&=&(u^*\vee v^*+I_{\Bn_{1/r}})^*\\
&=&((u^*+I_{\Bn_{1/r}})\vee(v^*+I_{\Bn_{1/r}}))^*\\
&=&((\reg{u}{r})^*\vee (\reg{v}{r})^*)^*\\
&=&\reg{u}{r}\wedge \reg{u}{r}.
\end{eqnarray*}
The second equation is proved analogously.
\end{proof}

\section{Inclusion-exclusion Properties of the Subgradient Map}\label{Inclusion-exclusion properties of the subgradient map}

In this section we will prove the following result.

\begin{theorem}\label{IE subgradient point} Let $u,v\in\fconvx$. If  $u\wedge v\in\fconvx$, then 
\begin{eqnarray*}
\partial(u\vee v)(x)\cup\partial(u\wedge v)(x)&=&\partial u(x)\cup\partial v(x),\label{alpha point}\\
\partial(u\vee v)(x)\cap\partial(u\wedge v)(x)&=&\partial u(x)\cap\partial v(x)\label{beta point}
\end{eqnarray*}
for every $x\in\R^n$.
\end{theorem}

The proof will require some preliminary steps. 

\begin{lemma}\label{lemma -1} Let $C_1$ and $C_2$ be non-empty, closed and convex subsets of $\R^n$. If $C_1\cup C_2$ is convex, then $C_1\cap C_2\ne\emptyset$.
\end{lemma}

\begin{proof} Assume that $C_1\cap C_2=\emptyset$. Let $B$ be a ball centered at the origin such that $K_i:=B\cap C_i\ne\emptyset$
for $i=1,2$. The convex bodies $K_1$ and $K_2$ are disjoint, and hence they can be strongly separated (see \cite[Theorem 1.3.7]{Schneider}) which implies in particular 
that their union is not convex. On the other hand $K_1\cup K_2=(C_1\cup C_2)\cap B$ must be convex. 
\end{proof}

Note that the claim of the previous lemma fails to be true without assuming the sets to be closed. 

\begin{lemma}\label{lemma 0} Let $u,v\in\fconvx$. If $u\wedge v\in\fconvx$ and  $u\vee v\ge0$ in $\R^n$, then
 $u\ge0$  or $v\ge0$ in $\R^n$. 
\end{lemma}

\begin{proof} By contradiction, assume that there exist $x_1,x_2\in\R^n$ such that $c:=u(x_1)\vee v(x_2)<0$. Let
$$
C_1=\{x\in\R^n\colon u(x)\le c\},\quad
C_2=\{x\in\R^n\colon v(x)\le c\}.
$$
These are non-empty, closed convex sets. Moreover
$$
\{x\in\R^n\colon (u\wedge v)(x)\le c\}=C_1\cup C_2,
$$
that is, their union is convex. By the previous lemma, there exists $x\in C_1\cap C_2$; on the other hand this leads to
$$
(u\vee v)(x)\le c<0,
$$
which is a contradiction.
\end{proof}

We proceed with a one-dimensional result.

\begin{lemma}\label{lemma 1} Let $u,v\colon\R\to\R\cup\{+\infty\}$ be convex and l.s.c.
and assume that $u\wedge v$ is convex. If $\bar x\in\R$ is such that
$u(\bar x)< v(\bar x)$, then there exists $\delta>0$ such that
$u(x)\le v(x)$ for all $x\in[\bar x-\delta,\bar x+\delta]$. 
\end{lemma}

\begin{proof} Clearly $\bar x\in\dom(u)$. If $\bar x$ is in the interior of $\dom(u)$, then $u$ is continuous at $\bar x$ and the statement follows. If $\bar x$ is a boundary point of $\dom(u)$ and $u(x)=+\infty$ for $x>\bar x$, say, then the convexity of $u\wedge v$ implies that $v(x)=+\infty$ for $x>\bar x$, too.
\end{proof}

\begin{lemma}\label{lemma 2} Let $u,v\in\fconvx$ and $u\wedge v\in\fconvx$. If $x\in\R^n$ 
is such that  $u(x)<v(x)$, then $(u\wedge v)'(x;y)=u'(x;y)$ for every $y\in\R^n$.
Moreover, if $v(x)$ is finite, then $(u\vee v)'(x;y)=v'(x;y)$ for every $y\in\R^n$.
\end{lemma}

\begin{proof} The statement is trivially true for $y=0$; so we assume $y\ne0$. 
Consider the line $L$ passing through $x$ and parallel to $y$. 
The restrictions of $u$ and $v$ to $L$ are l.s.c. convex functions of one variable, such that their minimum is convex. By the previous lemma, 
for $|t|$ sufficiently small,
$$
(u\wedge v)(x+ty)=u(x+ty),\quad (u\vee v)(x+ty)=v(x+ty).
$$
The conclusion follows immediately.
\end{proof}

\bigskip

\noindent{\em Proof of Theorem \ref{IE subgradient point}}. 
First, if $u(x)<v(x)$, then Lemma \ref{rock lemma} and Lemma \ref{lemma 2} imply that
\begin{equation}\label{subdifferential ineq}
\partial (u\wedge v)(x)=\partial u(x)\,\,\,\text{ and }\,\,\,\partial(u\vee v)(x)=\partial v(x).
\end{equation}
Second, if $u(x)=v(x)$,  we show that 
\begin{equation}\label{subdifferential eq}
\begin{array}{rcl}
\partial(u\vee v)(x)&=&\partial u(x)\cup\partial v(x),\\
\partial(u\wedge v)(x)&=&\partial u(x)\cap\partial v(x).
\end{array}
\end{equation}
Indeed, if $u(x)=v(x)=+\infty$, then all subdifferentials are empty. So assume that $x\in\dom(u)\cap\dom(v)$. 
The first equation in \eqref{subdifferential eq} is a straightforward consequence of the definition of subdifferentials and the equality $u(x)=v(x)$. Concerning 
the second equation, let
$$
C=\partial(u\vee v)(x),\quad D=\partial u(x)\cup\partial v(x).
$$
If $y\in D$, then we may assume that $y\in\partial u(x)$, which yields, for every $z\in\R^n$,
$$
(u\vee v)(z)\ge u(z)\ge u(x)+\langle z-x,y\rangle =(u\vee v)(x)+\langle z-x,y\rangle,
$$
that is,  $y\in C$. Now assume that $y\in C$. Define $w\in\fconvx$ as
$$
w(z)=u(x)+\langle z-x,y\rangle,
$$
and set $\bar u=u-w$ and $\bar v=v-w$. Then $\bar u,\bar v\in\fconvx$ and
$$
\bar u\wedge \bar v=(u\wedge v)-w,\quad \bar u\vee \bar v=(u\vee v)-w.
$$
In particular, $\bar u\wedge \bar v\in\fconvx$. Now, $y\in\partial(u\vee v)(x)$ implies $0\in\partial(\bar u\vee\bar v)(x)$; moreover
$(\bar u\vee \bar v)(x)=0$. Hence $\bar u\vee \bar v\ge0$ in $\R^n$. By Lemma \ref{lemma 0}, one of the two functions $\bar u$ and $\bar v$ is
non-negative in $\R^n$. Assuming for instance $\bar u\ge0$, we obtain, as $\bar u$ vanishes at $x$, that $0\in\partial\bar u(x)$ which implies $y\in\partial u(x)$,
that is, $y\in D$. 

Finally, note that (\ref{subdifferential ineq}) and (\ref{subdifferential eq}) imply the statement of the theorem.
\hfill$\square$\\

\section{The Graph of the Subgradient Map and  Parallel Sets of Functions}\label{The graph of the subgradient map and the parallel set of a function}

We start by recalling two important definitions.

\begin{definition}\label{nor} For $u\in\fconvx$, the graph of the subdifferential map  of $u$ is defined as
$$
\Gamma_u:=\{(x,y)\colon x\in\R^n,\ y\in\partial u(x)\}.
$$
\end{definition}

\goodbreak
\medskip
 Next we define parallel sets of a convex function. 

\begin{definition}\label{parallel set} Let $u\in\fconvx$ and $\eta\subset\R^n\times\R^n$. For $s\ge 0$, we set
$$
P_s(u,\eta)=\{x+s y\colon  (x,y)\in\eta\cap\Gamma_u\}.
$$
\end{definition}

Note that for $s=0$ we have $P_0(u,\eta)=\pi_1(\eta \cap \Gamma_u)$.

\subsection{Inclusion-exclusion results}

The following result is an immediate consequence of Theorem~\ref{IE subgradient point}.

\begin{proposition}\label{val-prop-E-Nor} Let $u,v\in\fconvx$. If  $u\wedge v\in\fconvx$, then
$$\Gamma_{u\vee v}\cap\Gamma_{u\wedge v}=\Gamma_u\cap\Gamma_v \,\,\,\text{ and }\,\,\,
\Gamma_{u\vee v}\cup\Gamma_{u\wedge v}=\Gamma_u\cup\Gamma_v.$$
\end{proposition}

Next, we establish a corresponding result for the parallel sets of the subdifferential graph.

\begin{proposition}\label{val-prop-Pi-rho} Let $u,v\in\fconvx$. If $u\wedge v\in\fconvx$, then
\begin{eqnarray}
P_s(u\vee v,\eta)\cup P_s(u\wedge v,\eta)=P_s(u,\eta)\cup P_s(v,\eta)\label{val-prop-Pi-rho1},
\\
P_s(u\vee v,\eta)\cap P_s(u\wedge v,\eta)=P_s(u,\eta)\cap P_s(v,\eta)\label{val-prop-Pi-rho2}
\end{eqnarray}
for every $\eta\subset\R^n\times\R^n$ and $s\ge0$. 
\end{proposition}

\begin{proof} The case $s=0$ follows easily from the previous proposition.
For $s>0$, we have
\begin{eqnarray*}
P_s(u\vee v,\eta)\cup P_s(u\wedge v,\eta)&=&\{x+s y\colon (x,y)\in\eta\cap\Gamma_{u\vee v}\}\cup
\{x+s y\colon (x,y)\in\eta\cap\Gamma_{u\wedge v}\}\\
&=&\{x+s y\colon (x,y)\in\eta\cap(\Gamma_{u\vee v}\cup\Gamma_{u\wedge v})\}\\
&=&\{x+s y\colon (x,y)\in\eta\cap(\Gamma_{u}\cup\Gamma_{v})\}\\
&=&\{x+s y\colon (x,y)\in\eta\cap\Gamma_{u}\}\cup\{x+s y\colon (x,y)\in\eta\cap\Gamma_{v}\}\\
&=&P_s(u,\eta)\cup P_s(v,\eta).
\end{eqnarray*}
This proves \eqref{val-prop-Pi-rho1}. The proof of \eqref{val-prop-Pi-rho2} is analogous.
\end{proof}

\subsection{An auxiliary proposition}

Given $u\in\fconvx$ and $r>0$, we set
$$
\Gamma_u^r:=\Gamma_u\cap\{(x,y)\in\R^{2n}\colon |y|\le r\}.
$$
We require the following properties of the Lipschitz regularization of a function $u\in\fconvx$ 
(in addition to those presented in Section \ref{Approximation}).

\begin{proposition}\label{help!} If $u\in\fconvx$ and $r>0$, then $u_r=\reg{u}{r}$  has the following properties.
\begin{enumerate}[(i)]
\item $\Gamma_u^r\subset\Gamma_{u_r}$. \label{two}
\item For every $\eta\subset\Gamma_u^r$ and $s\ge0$, we have $P_s(u,\eta)=P_s(u_r,\eta)$. \label{three}
\end{enumerate}
\end{proposition}

\begin{proof}
By \eqref{sechs} from Proposition \ref{approximation lemma}, it follows from $(x,y)\in\Gamma_u^r$ that $u_r(x)=u(x)$ and
$\partial u_r(x)=\partial u(x)\cap\{y\in\R^n\colon|y|\le r\}$, which implies \eqref{two} and therefore \eqref{three}. 
\end{proof}

\section{Hessian Measures}\label{Hessian measures}

We recall and extend the definition of Hessian measures. First, we establish the following result.

\begin{theorem}\label{Thm Hessian measures}
For $u\in\fconvx$, there are non-negative Borel measures $\Theta_0(u,\cdot), \dots, \Theta_n(u,\cdot)$ on $\R^{2n}$
such that
\begin{equation}\label{Steiner}
\hm^n(P_s(u,\eta))=\sum_{i=0}^n\binom nis^i\,\Theta_{n-i}(u,\eta)
\end{equation}
for every $\eta\in\Borel(\R^{2n})$ and $s\ge0$,
\end{theorem}

We will call the measures $\Theta_i(u,\cdot)$ for $i=0,\dots,n$ the {\em Hessian measures} of $u$.  The proof makes use of the following result from 
\cite{Colesanti-Hug-2000b} (see Theorem 3.1; also see \cite[Section 5]{Colesanti-Hug-2000}).
For an open subset $U$ of $\R^n$, we denote by $\Borel(U)$ the family of Borel subsets of $U$.

\begin{theorem}\label{CH}
If $u\in\fconvx$ and $U:=\interno(\dom(u))$ is not empty, then there are non-negative Borel measures $\Theta_0(u,\cdot), \dots, \Theta_n(u,\cdot)$ on $U\times \R^n$  
such that \eqref{Steiner} holds for every $\eta\in\Borel(U\times\R^n)$ and $s\ge0$.
\end{theorem}

\subsection{ Proof of Theorem \ref{Thm Hessian measures}.} 
Let $r>0$ and set $u_r =\reg{u}{r}$. 
For $\eta\in\Borel(\R^{2n})$ with $\eta\subset \R^n\times \Bn_r$, 
it follows from \eqref{three} of Proposition~\ref{help!}  that 
$P_s(u,\eta)=P_s(u_r, \eta\cap\Gamma_u)$. Hence, by Theorem \ref{CH},
\begin{equation*}
\hm^n(P_s(u,\eta))=\sum_{i=0}^n\binom nis^i\,\Theta_{n-i}(u_r, \eta\cap\Gamma_u).
\end{equation*}
In this way we have proved the following fact: for every $r>0$, there exists a set of $(n+1)$ Borel measures on $\Borel(\R^n\times \Bn_r)$,
namely $\Theta_{i,r}(u,\cdot\cap\Gamma_u) :=\Theta_{i}(u_r,\cdot\cap\Gamma_u)$, such that \eqref{Steiner} holds for every $\eta\in\Borel(\R^n\times \Bn_r)$ and $s\ge0$. If $r'\ge r\ge0$,
as $\Borel(\R^n\times \Bn_{r'})\supset\Borel(\R^n\times \Bn_r)$, we easily get
$$\Theta_{i,r}(u,\eta)=\Theta_{i,r'}(u,\eta)$$
for all $\eta\in\Borel(\R^n\times \Bn_r)$ and $i=0,\dots,n$. Such measures can be extended to $\Borel(\R^{2n})$ by a standard procedure, as they are non-negative:
$$
\Theta_i(u,\eta)=\lim\nolimits_{r\to+\infty}\Theta_{i,r}(u,\eta\cap (\R^n\times \Bn_r))
$$
for $\eta\in\Borel(\R^{2n})$ and $i=0,\dots,n$.

The validity of the Steiner formula \eqref{Steiner} is preserved as
$$
P_s(u,\eta)=\bigcup\nolimits_{r\ge 0}\ P_s(u,\eta\cap\Gamma_u^r)
$$
so that 
$$
\hm^n(P_s(u,\eta))=\lim_{r\to+\infty}\hm^n(P_s(u,\eta\cap (\R^n\times \Bn_r)))
$$
for every $\eta\in\Borel(\R^{2n})$ and  $s\ge 0$.
\hfill $\square$
\smallskip
\goodbreak

\begin{remark} The notion of Hessian measure is clearly of local nature. Let $u$ and $v$ be real-valued convex functions defined in an open convex set 
$U\subset\R^n$. Assume that $\partial u(x)=\partial v(x)$ for every $x\in\beta\in\Borel(U)$. If $\eta\in\Borel(U\times\R^n)$ is such that 
$\pi_1(\eta)\subset\beta$, then
$$
P_s(u,\eta)=P_s(v,\eta).
$$
for every $s\ge 0$. Hence

$$
\Theta_i(u,\eta)=\Theta_i(v,\eta)
$$
for $i=0,\dots,n$.
\end{remark}

\begin{remark}\label{support of Hessian measures} For $u\in\fconvx$, the support of the Hessian measures of $u$ is contained in $\Gamma_u$. Indeed,
if $\eta\subset\R^n\setminus\Gamma_u$, then for every $(x,y)\in\eta$, we have $\partial u(x)=\emptyset$. Hence, for $s\ge0$, we have $P_s(u,\eta)=\emptyset$, which implies that
$$
0=\hm^n(P_s(u,\eta))=\sum_{i=0}^n\binom nis^i\,\Theta_{n-i}(u,\eta).
$$
As Hessian measures are non-negative, we get $\Theta_i(u,\eta)=0$ for every $i=0,\dots,n$. 
\end{remark}

\begin{remark}\label{locally-finite-Hess-measures} Hessian measures are {\em locally finite}.
Indeed, assume that $|(x,y)|\le r$ for every $(x,y)\in\eta$, for some $r>0$. As each $\Theta_i(u,\cdot)$ is non-negative, by the Steiner formula 
(for $s=1$) we get
$$
\binom ni\ \Theta_i(u,\eta)\le\hm^n(P_1(u,\eta)).
$$
On the other hand, we have $P_1(u,\eta)\subset\{z\in\R^n\colon |z|\le 2r\}$.
Hence, for $i=0,\dots, n$,
\begin{equation}\label{serve}
\Theta_i(u,\eta)\le c(n)\ \diam(\eta)^{n}
\end{equation}
for every $u\in\fconvx$ and  $\eta\in\Borel(\R^{2n})$, where $c(n)$ only depends on the dimension and $\diam$ stands for diameter.
\end{remark}

\subsection{Continuity}

We show that Hessian measures are weakly continuous with respect to epi-convergence.

\begin{theorem}\label{continuity} Let $u_k$ be a sequence in $\fconvx$. If $u_k$ epi-converges to $u\in\fconvx$, then
the sequence of measures $\Theta_i(u_k,\cdot)$ converges weakly to $\Theta_i(u,\cdot)$ as $k\to +\infty$ for every $i=0,\dots,n$.
\end{theorem}

\begin{proof}
First, using the Steiner formula \eqref{Steiner}, it is easy to see that it is sufficient to prove for every $s>0$ that the sequence of measures $\eta\mapsto\hm^n(P_s(u_k,\eta))$
converges weakly to the measure $
\eta\mapsto\hm^n(P_s(u,\eta))$ as $k\to+\infty$.
\medskip

Second, to establish the previous claim we prove that if $\eta\subset\R^{2n}$ is compact, then
\begin{equation}\label{compact}
\limsup_{k\to+\infty}\hm^n(P_s(u_k,\eta))\le\hm^n(P_s(u,\eta)),
\end{equation}
and, if $\eta\subset\R^{2n}$ is open, then
\begin{equation}\label{open}
\liminf_{k\to+\infty}\hm^n(P_s(u_k,\eta))\ge\hm^n(P_s(u,\eta)).
\end{equation}

\medskip

Third, we prove \eqref{compact}. To simplify notation, for fixed $\eta\subset\R^{2n}$ and $s>0$, we set
$A=P_s(u,\eta)$ and $A_k=P_s(u_k,\eta)$ for $k\in\N$.
As $\eta$ is compact, $A$ is compact as well.
Indeed, it is obviously bounded. To prove that it is closed, let $z_j$ be a sequence in $A$, converging to some $\bar z\in\R^n$. 
Then there exists a sequence $x_j$ in $\R^n$ and a sequence $y_j$ such that $y_j\in\partial u(x_j)$ and $(x_j,y_j)\in\eta$ for every $j$ and 
$z_j=x_j+s y_j$. By compactness, we may assume that $x_j$ and $y_j$ converge to some $\bar x\in\R^n$ and $\bar y\in\R^n$, respectively, with
$(\bar x,\bar y)\in\eta$. For every $x\in\R^n$, we have
$$
u(x)\ge u(x_j)+\langle x-x_j,y_j\rangle.
$$
Passing to the limit and using the lower semicontinuity of $u$, we obtain
$$
u(x)\ge u(\bar x)+\langle x-\bar x,\bar y\rangle
$$
for all  $x\in\R^n$. Hence $\bar y\in\partial u(\bar x)$, so that $\bar z\in A$. 
\goodbreak

Next, we prove that for every $\varepsilon>0$ there exists $\bar k\in\N$ such that 
\begin{equation}\label{compact 3}
A_k\subset(A)_\varepsilon 
\end{equation}
for all $k\ge\bar k$, 
where $
A_\varepsilon:=\{x\in\R^n\colon \dist(x,A)\le\varepsilon\}$ and $\dist(x,A)=\inf\{\vert x-y\vert: y\in A\}$.
Note that \eqref{compact} follows from \eqref{compact 3}, as
$A=\bigcap_{\varepsilon>0}A_\varepsilon$ implies that $\hm^n(A)=\lim_{\varepsilon\to0^+}\hm^n(A_\varepsilon)$,
where the first equality holds because $A$ is closed. 
We argue by contradiction. Assume that there exists $\varepsilon>0$, a sequence $x_k\in\R^n$ and a sequence $y_k\in\R^n$, with 
$y_k\in\partial u_k(x_k)$ and $(x_k,y_k)\in\eta$ for every $k\in\N$, such that 
\begin{equation}\label{distance}
\dist(x_k+s y_k,A)\ge \varepsilon
\end{equation}
for all $k\in\N$.
By compactness, we may assume that the sequences $x_k$ and $y_k$ converge to $\bar x$ and $\bar y\in\R^n$, respectively,
with $(\bar x,\bar y)\in\eta$. Let $x\in\R^n$; by the definition of epi-convergence, there exists a sequence $\hat x_k$ such that
$$
\lim_{k\to+\infty} u_k(\hat x_k)=u(x).
$$
Moreover
$$
\liminf_{k\to+\infty} u_k(x_k)\ge u(\bar x),
$$
again by epi-convergence. From the inequality
$$
u_k(\hat x_k)\ge u_k(x_k)+\langle\hat x_k-x_k, y_k\rangle
$$
for all $k\in\N$,
passing to the limit in $k$, we deduce
$$
u(x)\ge u(\bar x)+\langle x-\bar x,\bar y\rangle
$$
for $x\in\R^n$, that is,
$\bar y\in\partial u(\bar x)$ so that $x_k+s y_k$ converges to a point of $A$, which contradicts \eqref{distance}. 
\goodbreak
\medskip

Finally, we prove \eqref{open}. Define $w, w_k\in\fconvx$ as
$$
w(x)=s\, u(x)+\frac{|x|^2}{2},\quad
w_k(x)=s\, u_k(x)+\frac{|x|^2}{2},
$$
and note that
$
\partial w(x)=x+s\, \partial u(x)$ while $\partial w_k(x)=x+s\, \partial u_k(x)$.
Let us fix $\bar x\in\R^n$ and $\bar z\in\partial w(x)$. The function $\hat w\in\fconvx$ defined as
$$
\hat w(x)=w(x)-\langle x-\bar x,\bar z\rangle
$$ 
has an absolute minimum at $x=\bar x$. Consider now, for $k\in\N$, the function $\hat w_k\in\fconvx$ defined as
$$
\hat w_k(x)=w_k(x)-\langle x-\bar x,\bar z\rangle.
$$ 
This is a strictly convex and lower semicontinuous function on $\R^n$, which verifies
$$
\lim_{|x|\to+\infty}\hat w_k(x)=+\infty.
$$
This follows from the fact that $u_k$ is bounded from below by an affine function, since its subdifferential is non-empty at least at one point. Hence 
$w_k$ admits a unique (by strict convexity) absolute minimum point $x_k$. This implies $0\in\partial\hat w_k(x_k)$ which in turn implies that $\bar z\in\partial w_k(x_k)$.

As $u_k$ epi-converges to $u$, the sequence $\hat w_k$ epi-converges to $\hat w$, so that (see for instance Theorem 7.33 in \cite{Rockafellar-Wets})
$$
\lim_{k\to+\infty} x_k=\bar x.
$$ 
We have proved that given $(\bar x,\bar y)\in\eta\cap\Gamma_u$, there exists a sequence $(x_k,y_k)$ with 
$(x_k,y_k)\in\Gamma_{u_k}$ for every $k$, such that $\bar x+s \bar y=x_k+s y_k$ for every $k\in\N$; 
moreover, $x_k$ tends to $\bar x$ as $k\to+\infty$. These conditions
imply that $y_k$ tends to $\bar y$. Hence, as $\eta$ is open, $(x_k,y_k)$ is definitively contained in $\eta$. Thus 
every $z\in P_s(u,\eta)$ is  contained in $P_s(u_k,\eta)$. Consequently,
$$
P_s(u,\eta)\subset\bigcup\nolimits_{i\in\N}\left( \bigcap\nolimits_{k\ge i}\ P_s(u_k,\eta)\right).
$$
This easily implies \eqref{open}.
\end{proof}

\begin{remark} In the special case of finite functions, the previous theorem was proved in \cite[Theorem 1.1]{Colesanti-Hug-2000}.
\end{remark}

\subsection{Integral representations of Hessian measures}

Our starting point is the following result, which follows in a rather direct way from results proved in \cite{Colesanti-Hug-2000b}. 
We recall that a subset $A$ of $\R^{2n}$ is {\em countably $n$-rectifiable} if there exist a countable family 
$\{U_i\colon i\in\N\}$ of open subsets of $\R^n$ and a countable family $\{f_i\colon i\in\N\}$ of Lipschitz maps, such that
$f_i\colon U_i\to\R^{2n}$ and
$$
\hm^n\left(A\setminus \bigcup\nolimits_{i\in\N}f_i(U_i)\right)=0
$$
(see {\em e.g.} \cite{Federer}).

\begin{proposition}\label{proposition integral representation} If $u\in\fconvx$ is such that $\dom(u)=\R^n$, then $\Gamma_u$ is a 
 countably $n$-rectifiable set. Moreover, for $i\in\{0,\dots,n\}$,
there exists a non-negative, $\hm^n$-measurable function $\theta_i(u,\cdot)\colon\Gamma_u \to \R$
such that 
\begin{equation*}
\Theta_i(u,\eta)=\int_{\eta\cap\Gamma_u}\theta_i(u,(x,y))\d\hm^{n}(x,y)
\end{equation*}
for every $\eta\in\Borel(\R^{2n})$.
\end{proposition}

\begin{proof} The countable $n$-rectifiability of $\Gamma_u$ was observed in \cite[Section 3]{Colesanti-Hug-2000b}. 

The set $\epi(u)$ is a non-empty, closed, convex subset 
of $\R^{n+1}$, which does not coincide with $\R^{n+1}$ itself. In particular, its {\em normal bundle}, 
$\nor(\epi(u))$, is well defined:
$$
\nor(\epi(u))=\{(p,q)\in\R^{n+1}\times\R^{n+1}\colon p\in\bd(\epi(u)),\ q\in N(\epi(u),p)\},
$$
where $N(\epi(u),p)$ denotes the outer normal cone to $\epi(u)$ at $p$. Finally, we set
$$
{\mathcal F}_u=\nor(\epi(u))\cap \left(\{(x,u(x))\colon x\in\R^n\}\times{\mathbb S}^n\right).
$$
\goodbreak
In \cite[Section 3]{Colesanti-Hug-2000b} the following facts are proved:
\begin{enumerate}[(i)]
\item ${\mathcal F}_u$ is a countably $n$-rectifiable Borel set;
\item there exists an homeomorphism $T\colon\Gamma_u \to {\mathcal F}_u$; moreover $T$ and its inverse are locally
Lipschitz.
\end{enumerate}

By \cite[Theorem 3.1]{Colesanti-Hug-2000b}, each Hessian measure of $u$ admits an integral representation on ${\mathcal F}_u$: there exist $\hm^n$-measurable, non-negative functions $\bar\theta_i(u,\cdot)\colon{\mathcal F}_u \to \R$, such that
\begin{equation}\label{int rep 2}
\Theta_i(u,\eta)=\int_{\bar\eta\cap{\mathcal F}_u}\bar\theta_i(u,(\bar x, \bar y)\d\hm^n(\bar x, \bar y),
\end{equation}
for every $\eta\in\Borel(\R^{2n})$ and $i\in\{0,\dots,n\}$, where
$\bar\eta=T(\eta\cap \Gamma_u)$. Now \eqref{int rep} follows from performing the change of variables $(\bar x, \bar y)=T(x,y)$ in \eqref{int rep 2} and using the co-area formula (see for instance 
\cite{Federer}).
\end{proof}

\goodbreak
Next we extend the previous result to the general case.

\begin{theorem}\label{thm integral representation} If $u\in\fconvx$, then $\Gamma_u$ is a countably $n$-rectifiable set. Moreover,  for  $i\in\{0,\dots,n\}$,
there exists a non-negative, $\hm^n$-measurable function $\theta_i(u,\cdot)\colon\Gamma_u\ \rightarrow\, \R$
such that
\begin{equation}\label{int rep}
\Theta_i(u,\eta)=\int_{\eta\cap\Gamma_u}\theta_i(u,(x,y))\d\hm^{n}(x,y).
\end{equation}
for every $\eta\in\Borel(\R^{2n})$.
\end{theorem}

\begin{proof} 
For $r>0$, set $u_r=\reg{u}{r}$ and note that  $\dom(u_r)=\R^n$. By Proposition \ref{help!}, $\Gamma^{1/r}_u\subset\Gamma_{u_r}$.
By Proposition \ref{proposition integral representation}, the set $\Gamma_{u_r}$ is countably $n$-rectifiable, and hence the same 
conclusion is valid for $\Gamma^{1/r}_u$. As
$$
\Gamma_u=\bigcup\nolimits_{r>0}\Gamma_u^{1/r}
$$
we deduce that $\Gamma_u$ is countably $n$-rectifiable as well.

Let us fix $i\in\{0,\dots,n\}$.  As in the proof of Theorem \ref{Thm Hessian measures}, we have 
$P_s(u,\eta)=P_s(u_r,\eta)$ for every Borel set $\eta\subset \Gamma_u\cap(\R^n\times \Bn_r)$.
So, in particular 
$$
\Theta_i(u,\eta)=\Theta_i(u_r,\eta)
$$
for every $\eta\in\Borel(\R^{2n})$ contained in $\Gamma_u\cap(\R^n\times \Bn_r)$.
By Proposition \ref{proposition integral representation}, there exists a non-negative $\hm^n$-measurable function $\theta_i(u_r,\cdot)$ defined on 
$\Gamma_{u_r}$ such that
$$
\Theta_i(u,\eta)=\Theta_i(u_r,\eta)=\int_{\eta\cap\Gamma_{u_r}}\theta_i(u_r,(x,y))\d\hm^n(x,y)=
\int_{\eta\cap\Gamma_u}\theta_i(u_r,(x,y))\d\hm^n(x,y).
$$
Clearly for $r'\ge r$, the function $\theta_i(u_{r'},\cdot)$ extends $\theta_i(u_r,\cdot)$. Hence there is a function $\theta_i(u,\cdot)$, defined on $\Gamma_u$, such that
\eqref{int rep} holds for every $\eta\in\Borel(\R^{2n})$ contained in $\R^n\times \Bn_r$. The conclusion follows from letting $r\to+\infty$.
\end{proof}

\subsection{A second auxiliary result}

For future use, we state the following result that follows from the argument used in the proofs in the previous part of this section.

\begin{proposition}\label{help!!} Let $u\in\fconvx$. For $r>0$, the function $u_r =\reg{u}{r}$ has the following
properties.
\begin{enumerate}[\emph(i)]
\item For every $\eta\subset\Gamma_u^r$ and $i\in\{0,\dots,n\}$,
$$\Theta_i(u,\eta)=\Theta_i(u_r,\eta).$$
\item For $\hm^{n}$-a.e. $(x,y)\in\Gamma_u^r$ and $i\in\{0,\dots,n\}$, 
$$
\theta_i(u,(x,y))=\theta_i(u_r,(x,y)).
$$
\end{enumerate}
\end{proposition}

\section{Further Properties of Hessian Measures}\label{Further properties of Hessian measures}

\subsection{Invariance and covariance properties}


Let $u\in\fconvx$ and $x_0\in\R^n$; we denote by $u_{x_0}$ the composition of $u$ and the translation by $x_0$; that is,
$$
u_{x_0}(x)=u(x-x_0)
$$
for $x\in\R^n$. Clearly, $u_{x_0}\in\fconvx$. 
Let $\phi$ be an element of $\O(n)$. Given $u\in\fconvx$, we define $u_{\phi}\colon\R^n\to \R^n$ by
$$
u_{\phi}(y)=u(\phi^{-1}y)
$$
for $y\in\R^n$. As before, we have $u_{\phi}\in{\fconvx}$.  For $\phi\in\O(n)$, we also define $\tilde \phi\colon\R^n\times\R^n\to\R^n\times\R^n$ by $\tilde \phi(x,y)=(\phi x,\phi y)$.

\medskip
We establish the following covariance properties of Hessian measures.

\begin{proposition}\label{behavior Hessian measures}
Let $u\in{\fconvx}$. For $x_0\in\R^n$ and  $\phi\in\O(n)$,
\begin{eqnarray*}
&&\Theta_i(u_{x_0},\eta)=\Theta_i(u,\eta+(x_0,0)),\\
&&\Theta_i(u_\phi,\eta)=\Theta_i(u,{\tilde \phi}^{-1}\eta).
\end{eqnarray*}
for every $\eta \in\Borel(\R^n\times\R^n)$ and  $i\in\{0,\dots,n\}$.
\end{proposition}

\begin{proof}  For every $x\in\R^n$, the definition of the subdifferential implies that $\partial u_{x_0}(x)=\partial u(x-x_0)$ and therefore
$\Gamma_{u_{x_0}}=\Gamma_u+(x_0,0)$.
Consequently,
\begin{eqnarray*}
P_s(u_{x_0},\eta)&=&\{x+s y\colon (x,y)\in\eta\cap\Gamma_{u_{x_0}}\}\\
&=&\{x+s y\colon (x,y)\in\eta \cap (\Gamma_{u}+(x_0,0))\}\\
&=&\{x+s y\colon (x,y)\in((\eta-(x_0,0))\cap \Gamma_u)+(x_0,0)\}\\
&=&\{x+s y\colon (x,y)\in(\eta-(x_0,0))\cap \Gamma_u\}+x_0\\
&=&P_s(u,\eta-(x_0,0))+x_0.
\end{eqnarray*}
Similarly,  $\phi^{-1} \partial u_\phi(x)=\partial u(\phi^{-1}x)$ and therefore $\Gamma_{u_\phi}=\tilde \phi(\Gamma_u)$, which implies that
\begin{eqnarray*}
P_s(u_\phi,\eta)&=&\{x+s y\colon (x,y)\in\eta \cap \Gamma_{u_\phi}\}\\
&=&\{x+s y\colon (x,y)\in \eta\cap\tilde \phi\, \Gamma_{u}\}\\
&=&\{x+s y\colon (x,y)\in\tilde \phi({\tilde \phi}^{-1}\eta \cap \Gamma_{u})\}\\
&=&\phi(P_s(u,{\tilde \phi}^{-1}\eta)).
\end{eqnarray*}
The conclusions follow from taking the Lebesgue measure of $P_s(u_{x_0},\eta)$ and of $P_s(u_\phi,\eta)$, respectively, and from using the Steiner type formula
\eqref{Steiner} and the invariance of the Lebesgue measure under rotations and translations. 
\end{proof}


\subsection{Hessian measures of $u$ and $u^*$}
Let $T\colon\R^{2n}\to\R^{2n}$ be defined by $T(x,y)=(y,x)$.
Given a subset $\eta$ of $\R^{2n}$, we set
$$
\hat\eta=\{(x,y)\in\R^n\times\R^n\colon (y,x)\in\eta\}=T(\eta).
$$
If $u\in\fconvx$, then its conjugate function $u^*$ belongs to $\fconvx$ as well. Moreover, by Lemma~\ref{le:conjugate_subgradient}, for every $(x,y)\in\R^n\times\R^n$, we have
$y\in\partial u(x)$ if and only if $x\in\partial u^*(y)$.
Hence
$$
\Gamma_u=T(\Gamma_{u^*}).
$$
Similarly,
$$
\Gamma_u\cap\eta=T(\Gamma_{u^*}\cap\hat\eta)
$$
for all $\eta\in\Borel(\R^{2n})$. We easily deduce the following result, already observed in \cite[Theorem 5.8]{Colesanti-Hug-2000}.

\begin{theorem}\label{duality} For $u\in\fconvx$, 
$$
\Theta_i(u,\eta)=\Theta_{n-i}(u^*,\hat\eta)
$$
for every $i=0,\dots,n$ and  $\eta\in\Borel(\R^{2n})$.
\end{theorem}

\subsection{The measures $\Theta_0$ and $\Theta_n$}  We consider the extremal cases $i=0$ and $i=n$ for Hessian measures.

\begin{proposition}\label{extreme cases} For every $u\in\fconvx$, 
$$ 
\Theta_n(u,\eta)=\hm^n(\pi_1(\eta\cap\Gamma_u)),\quad
\Theta_0(u,\eta)=\hm^n(\pi_2(\eta\cap\Gamma_u)),\quad
$$ 
for all $\eta\in\Borel(\R^{2n})$.
\end{proposition}

\begin{proof} We have $P_0(u,\eta)=\pi_1(\eta \cap\Gamma_u)$.
Hence  the first equality follows from \eqref{Steiner} with $s=0$. The second is obtained by applying the first one to $u^*$ and by Theorem \ref{duality}.
\end{proof}

A particular case of the previous result was already presented in \cite[p.\ 3248]{Colesanti-Hug-2000b}.

\subsection{Lower dimensional domains}

\begin{proposition}\label{theorem low dimension} 
If $\,u\in\fconvx$, then $\Theta_i(u,\cdot)\equiv0$ for every $i>\dimension(\dom(u))$. 
\end{proposition}

\begin{proof} Let $\dimension(\dom(u))=k$. There exists an affine subspace $E$ of $\R^n$ of dimension $k$, such that $\dom(u)\subset E$. Without loss of generality, we may assume
that
$$
E=\{(x',0)\colon x'\in\R^k\}, 
$$
where $0$ stands for the zero element of $\R^{n-k}$. Let $F$ be the orthogonal complement of $E$, that is,
$$
F=\{(0,x'')\colon x''\in\R^{n-k}\}.
$$
Define $u_E\colon\R^k\to\R\cup\{+\infty\}$ as the restriction of $u$ to $E$, that is, $u_E(x')=u(x',0)$ for $x'\in\R^k$. It follows 
from $\dom(u)\subset E$ that $u_E\in\fconvxk$. It is not hard to check that, for every $x'\in\R^k$,
$$
\partial u(x',0)=\{(y',y'')\colon y'\in\partial u_E(x'),\ y''\in\R^{n-k}\}.
$$
Hence 
$$
\Gamma_u=\{((x',0),(y',y''))\colon x'\in\R^k,\ y'\in\partial u_E(x'),\ y''\in\R^{n-k}\}. 
$$
Now let $\alpha',\beta'\in\Borel(\R^k)$ and  $\alpha'', \beta''\in\Borel(\R^{n-k})$ with $0\in\alpha''$. Set
$$
\eta=\alpha'\times\alpha''\times\beta'\times\beta''\in\Borel(\R^{2n}).
$$
For $s\ge0$, we obtain
\begin{eqnarray*}
P_s(u,\eta)&=&\{(x'+s y',s y'')\colon (x',y')\in\Gamma_{u_E}\cap(\alpha'\times\beta'),\ y''\in\beta''\}\\
&=&P_s(u_E,\alpha'\times\beta')\times s\beta''.
\end{eqnarray*}
Taking Lebesgue measures on both sides, we get
$$
\hm^n(P_s(u,\eta))=s^{n-k}\ \hm^k(P_s(u_E, \alpha'\times\beta')).
$$
This implies that the polynomial expansion of $\hm^n(P_s(u,\eta))$ does not contain any term of order lower than $(n-k)$. Thus
$$
\Theta_{k+1}(u,\eta)=\dots=\Theta_{n}(u,\eta)=0.
$$
The conclusion now follows from the arbitrariness of $\alpha', \alpha'', \beta', \beta''$ and the non-negativity of Hessian measures.
\end{proof}

\goodbreak
\subsection{Smooth functions}\label{smooth functions} For $u\in\fconvx\cap C^2(\R^n)$, we have
\begin{equation}\label{Gamma for smooth functions}
\Gamma_u=\{(x,\nabla u(x))\colon x\in\R^n\},
\end{equation}
and
$$
P_s(u,\beta\times \R^n)=\{x+s \,\nabla u(x)\colon x\in \beta \times \R^n\}
$$
for every $\beta\in\Borel(\R^{n})$ and $s\ge0$.
Hence
\begin{eqnarray*}
\hm^n(P_s(u, \beta\times \R^n))&=&\int_{P_s(u,\beta\times \R^n)}\d z
=\int_{\beta}\det(I_n+s\Hess u(x))\d x,
\end{eqnarray*}
by the change of variables $z=x+s\,\nabla u(x)$. On the other hand:
$$
\det(I_n+s \Hess u(x))=\sum_{i=0}^n\binom ni s^i\, [\Hess u(x)]_i.
$$ 
By Theorem \ref{Thm Hessian measures}, we get
\begin{equation}\label{Hessian measures for smooth functions}
\Theta_i(u,\beta\times\R^n)=\int_{\beta} [\Hess u(x)]_{n-i}\d x
\end{equation}
for every $i\in\{0,\dots, n\}$ and $\beta\in\Borel(\R^{n})$.

\section{Valuations on Convex Functions}\label{Valuations}

Let $\langle{\G},+\rangle$ be an Abelian semi-group. 

\begin{definition} A function $\oZ\colon \fconvx\to{\G}$  is called a valuation if 
$$
\oZ(u\vee v)+\oZ(u\wedge v)=
\oZ(u)+\oZ(v)
$$
for every $u,v\in\fconvx$ such that $u\wedge v\in\fconvx$. 
\end{definition}

We will only be interested in two cases: ${\G}=\R$ and $\G={\mathcal M(\R^{2n})}$; the latter is the space of Borel measures 
on $\R^{2n}$. In both situations there exists a natural choice of  topology, the Euclidean topology on $\R$ and the topology
of weak convergence on ${\mathcal M(\R^{2n})}$. A valuation $\oZ$ is called continuous if, for every sequence $u_k$ in $\fconvx$
epi-convergent to $u\in\fconvx$, we have
$\lim_{k\to+\infty}\oZ(u_k)=\oZ(u)$ in $\G$.

\medskip

In the following definitions we use the notation introduced in the previous section for the composition of a function with a translation or a rotation.
Let $\oZ\colon\fconvx\to\R$ be a valuation. We say that 
\begin{enumerate}[(i)]
\item $\oZ$ is translation invariant if $\oZ(u_{x_0})=\oZ(u)$ for all $u\in\fconvx$ and $x_0\in\R^n$;
\item $\oZ$ is rotation invariant if
$\oZ(u_\phi)=\oZ(u)$ for all $u\in\fconvx$ and $\phi\in\O(n)$;
\item $\oZ$ is rigid motion invariant if it is translation and rotation invariant.
\item $\oZ$ is $i$-simple, for $i=1,\dots,n$, if $\oZ(u)=0$ whenever $u\in\fconvx$ and $\dim(\dom(u))\le i+1$.
\end{enumerate}

\medskip
\noindent
Clearly, all the previous definition can be repeated for functionals defined on subsets of $\fconvx$.

\subsection{Measure-valued valuations}\label{Measure-valued valuations}

\begin{theorem}\label{Hessian measure valuations} 
For $i\in\{0,\dots,n\}$, the function $\Theta_i\colon\fconvx\to{\mathcal M(\R^{2n})}$ is a continuous and $i$-simple valuation.
\end{theorem}

\begin{proof} Continuity and simplicity follow from Theorems \ref{continuity} and \ref{theorem low dimension}, respectively.
As for the valuation property, let $u,v\in\fconvx$ be such that $u\wedge v\in\fconvx$. We have to prove that
$$
\Theta_i(u\wedge v,\eta)+\Theta_i(u\vee v,\eta)=
\Theta_i(u,\eta)+\Theta_i(v,\eta)
$$
for $\eta\in\Borel(\R^{2n})$.
By Proposition \ref{val-prop-Pi-rho} we have, for every $s\ge0$,
\begin{align*}
\hm^n(P_s&(u\wedge v,\eta))+\hm^n(P_s(u\vee v,\eta))\\
&=\hm^n(P_s(u\wedge v,\eta)\cup P_s(u\vee v,\eta))+\hm^n(P_s(u\wedge v,\eta)\cap P_s(u\vee v,\eta))\\
&=\hm^n(P_s(u,\eta)\cup P_s(v,\eta))+\hm^n(P_s(u,\eta)\cap P_s(v,\eta))\\
&=\hm^n(P_s(u,\eta))+\hm^n(P_s(v,\eta)).
\end{align*}
The conclusion follows immediately from the Steiner type formula \eqref{Steiner}.
\end{proof}

\section{Hessian Valuations}\label{Hessian valuations}

Let $\zeta\colon\R\times\R^n\times\R^n\ \rightarrow\ \R$ be a continuous function. We extend its definition to include the value $+\infty$ for its 
first variable, setting $\zeta(+\infty,x,y)=0$  for all $(x,y)\in\R^n\times\R^n$.
This extension will be always tacitly used in the rest of the paper. 

\goodbreak
We are now in a position to prove Theorem \ref{theorem intro 1} and first establish the following result.

\begin{theorem}\label{theorem Hessian valuations} Let $\zeta\in C(\R\times\R^n\times\R^n)$ have compact support with respect to the second and third variables. For $i\in\{0,1,\dots,n\}$,
the functional $\oZ_{\zeta,i}\colon {\fconvx}\to\R$, defined by
\begin{equation}\label{Hessian valuation - second type}
\oZ_{\zeta, i}(u)=\int_{\Gamma_u}\zeta(u(x),x,y)\d\Theta_i(u,(x,y)),
\end{equation}
is an $i$-simple and continuous valuation on $\fconvx$.
\end{theorem}

\goodbreak
\begin{remark}
If $u\in\fconvx\cap C^2(\R^n)$, then, by \eqref{Gamma for smooth functions} and  \eqref{Hessian measures for smooth functions},
$$
\oZ_{\zeta, i}(u)=\int_{\Gamma_u}\zeta(u(x),x,y)\d\Theta_i(u,(x,y))
=\int_{\R^n}\zeta(u(x),x,\nabla u(x))\,[\Hess u(x))]_{n-i}\d x.
$$
Hence, all statements of Theorem \ref{theorem intro 1} follow from the above theorem.
\end{remark}

\goodbreak
The proof of Theorem \ref{theorem Hessian valuations} requires some preparatory steps.

\begin{lemma}\label{continuity lemma} Let $\zeta\in C(\R\times\R^n\times\R^n)$ and  $u\in\fconvx$. 
The function $\xi\colon\Gamma_u\to\R$, defined as $\xi(x,y)=\zeta(u(x),x,y)$, 
is continuous on $\Gamma_u$.
\end{lemma}

\begin{proof} 
For $r>0$, by Proposition \ref{approximation lemma}, the function $u_r=\reg{u}{r}$ has $\dom(u_r)=\R^n$ and 
$u(x)=u_r(x)$ for $x$ such that $\vert y\vert \le 1/r$ for every $y\in\partial u(x)$.
Hence 
$$
\xi(x,y)=\zeta(u(x),x,y)=\zeta(u_r(x),x,y)
$$
for $(x,y)\in\Gamma_u^r$. As $u_r$ is continuous in $\R^n$, $\xi$ is continuous on $\Gamma_u^r$.
As $r>0$ is arbitrary, the conclusion follows.
\end{proof}

\begin{proposition}\label{summability}
Let $\zeta\in C(\R\times\R^n\times\R^n)$ have compact support with respect to the second and third variables. For every  $u\in\fconvx$, the function $\xi\colon\Gamma_u \to\R$, defined as
$
\xi(x,y)=\zeta(u(x),x,y),
$
is summable with respect to the measure $\Theta_i(u,\cdot)$ for $i\in\{0,\dots,n\}$.
\end{proposition}

\begin{proof} We recall that the support of $\Theta_i(u,\cdot)$ is contained in $\Gamma_u$. Set  $\zeta_+=\max\{\zeta,0\}$ and  $\zeta_-=\min\{\zeta,0\}$.
We clearly have 
$$
\xi_+(x,y):=\max\{\xi(x,y),0\}=\zeta_+(u(x),x,y),\quad
\xi_-(x,y):=\min\{\xi(x,y),0\}=\zeta_-(u(x),x,y).
$$
By Lemma \ref{continuity lemma},  the function $\xi_+$ is continuous and  hence integrable with respect
to the Borel measure $\Theta_i(u,\cdot)$ over $\Gamma_u$.  
Let us prove that $\xi_+$ is in fact summable (the proof for $\xi_-$ is completely analogous).
As $\zeta$ has compact support in $(x,y)$, there exists $r>0$ such that $\zeta(t,x,y)=0$ whenever $|x|\ge r$ or $|y|\ge r$. Moreover, there exists 
$u_r\in\fconvx$ such that $\dom(u_r)=\R^n$, and
$u(x)=u_r(x)$ for all $x\in\pi_1(\Gamma_u^r)$.
Hence
$$
\xi(x,y)=\zeta(u_r(x),x,y)$$
for all $(x,y)\in\Gamma_u$.
By continuity, $u_r$ is bounded in the set $\{x\in\R^n\colon |x|\le r\}$ and there exists a constant $c>0$ such that
$$
\xi_+(x,y)\le c
$$
for all $(x,y)\in\Gamma_u$.
Consequently,
$$
\int_{\R^{2n}}\zeta(u_r(x),x,y)\d\Theta_i(u_r,(x,y))\le\ c\ \Theta_i(u_r,\eta) 
$$
for $\eta=\{(x,y)\in\R^{2n}\colon |x|,\ |y|\le r\}$. The conclusion follows from the fact that $\Theta_i(u_r,\cdot)$ is locally finite
(see \eqref{serve}).
\end{proof}

\goodbreak

\subsection{Proof of Theorem \ref{theorem Hessian valuations}.} 
By Proposition \ref{summability}, the integral \eqref{Hessian valuation - second type} is finite.

For simplicity, for an arbitrary $u\in\fconvx$ we write $\Theta_i(u,\cdot)$ as $\Theta(u,\cdot)$. As remarked above,
the support of $\Theta(u,\cdot)$ is contained in $\Gamma_u$. Moreover, by Theorem \ref{thm integral representation}, we have
\begin{equation}\label{valuation-Hess-measure-1}
\int_{\Gamma_u}\zeta(u(x),x,y)\d\Theta(u,(x,y))=\int_{\Gamma_u}\zeta(u(x),x,y)\theta(u,(x,y))\d\hm^n(x,y),
\end{equation}
for a suitable non-negative, $\hm^n$-measurable function $\theta(u,\cdot)$ defined on $\Gamma_u$. 
Let $r>0$ be such that $\zeta(t,x,y)=0$ if either $|x|>r$ or $|y|>r$. For $u\in\fconvx$, we set $u_r =\reg{u}{r}$ and use 
the properties established in Proposition \ref{help!} and Proposition \ref{help!!}. 

\medskip

\noindent
{\bf 1. The valuation property.} Let $u, v\in\fconvx$ be such that $u\wedge v\in\fconvx$. Set 
$$
\Sigma=\Gamma_u\cup\Gamma_v=\Gamma_{u\vee v}\cup\Gamma_{u\wedge v}
$$
(see Proposition \ref{val-prop-E-Nor}). We will consider a generic point $(x,y)\in\Sigma$ such that the densities
$$
\theta(u,(x,y)), \theta(v,(x,y)), \theta(u\vee v,(x,y)), \theta(u\wedge v,(x,y))
$$ 
are well defined; this happens to be true
for $\hm^{n}$-a.e. $(x,y)\in\Sigma$. In view of \eqref{valuation-Hess-measure-1}, it is enough to prove that for $\hm^{n}$-a.e.
$(x,y)\in\Sigma$, we have
\begin{align}\label{tesi}
\zeta(u(x),&\, x,y)\,\theta(u,(x,y))+\zeta(v(x),x,y)\,\theta(v,(x,y))\nonumber \\[-6pt] 
\\[-6pt]
&=\zeta((u\vee v)(x),x,y)\,\theta(u\vee v,(x,y))+\zeta((u\wedge v)(x),x,y)\,\theta(u\wedge v,(x,y)).\nonumber
\end{align}
Indeed, integrating over $\Sigma$ we get
$$
\oZ_{\zeta,i}(u)+\oZ_{\zeta,i}(v)=\oZ_{\zeta,i}(u\vee v)+\oZ_{\zeta,i}(u\wedge v).
$$
It will suffice to prove \eqref{tesi} under the assumption $|y|\le r$, as it is trivially true for $\vert y\vert >r$ since $\zeta$ vanishes
at each point where it is computed. 

By the valuation property of Hessian measures (Theorem \ref{Hessian measure valuations}), we deduce
\begin{equation}\label{ipotesi}
\theta(u\vee v,(x,y))+\theta(u\wedge v,(x,y))=
\theta(u,(x,y))+\theta(v,(x,y))\quad\mbox{for $\hm^{n}$-a.e. $(x,y)\in\Sigma$.}
\end{equation}
This shows in particular that \eqref{tesi} holds true for every $(x,y)$ such that $u(x)=v(x)$. Next assume that
$(x_0,y_0)$ is such that $u(x_0)\ne v(x_0)$ and without loss of generality, assume that $u(x_0)> v(x_0)$. 

{\em Case 1}: Let $(x_0,y_0)\notin\Gamma_u\cap\Gamma_v$. Then we also have, by  Proposition \ref{val-prop-E-Nor},
$(x_0,y_0)\notin\Gamma_{u\vee v}\cap\Gamma_{u\wedge v}$. Assume that $(x_0,y_0)\in\Gamma_u\setminus\Gamma_v$. 
From  \eqref{subdifferential ineq}, $y_0\in\partial u(x_0)=\partial (u\vee v)(x_0)$; therefore
$$
(x_0,y_0)\in\Gamma_{u\vee v}\setminus\Gamma_{u\wedge v}.
$$
Hence
$$
\theta(v,(x_0,y_0))=\theta(u\wedge v, (x_0,y_0))=0.
$$
In this case \eqref{ipotesi} reduces to
$$
\theta(u\vee v,(x_0,y_0))=\theta(u,(x_0,y_0))
$$
and \eqref{tesi} follows from multiplying both sides of the previous equation 
by $(u\vee v)(x_0)=u(x_0)$. The case $(x_0,y_0)\in\Gamma_v\setminus\Gamma_u$ is completely analogous.

{\em Case 2}: Let $(x_0,y_0)\in\Gamma_u\cap\Gamma_v$. Then 
$u_r(x_0)=u(x_0)>v(x_0)=v_r(x_0)$. Let $U$ be a  
neighborhood of $x_0$ such that $u_r>v_r$ in $U$ (which exists by the continuity of $u_r$ and $v_r$). Hence
$$
u_r\wedge v_r\equiv v_r,\quad u_r\vee v_r\equiv u_r\quad\mbox{in  $\,U$.}
$$
Then we also have $\partial(u_r\wedge  v_r)(x)=\partial v_r(x)$ and $\partial(u_r\vee v_r)(x)=\partial v_r(x)$ for every $x\in{U}$,
and this implies that
$$
P_s(u_r\wedge v_r,\eta)=P_ s(v_r,\eta),\quad 
P_s(u_r\vee v_r,\eta)=P_s(u_r,\eta)
$$
for every $s\ge0$ and for every $\eta\subset\R^n\times\R^n$ such that $\pi_1(\eta)\subset{U}$. Consequently,
$$
\Theta(u_r\wedge v_r,\eta)=\Theta(v_r,\eta),\quad 
\Theta(u_r\vee v_r,\eta)=\Theta(u_r,\eta),
$$
for every such $\eta$. 
We deduce that
$$
\theta(u_r\wedge v_r,x_0)=\theta(v_r,x_0),\quad 
\theta(u_r\vee v_r,x_0)=\theta(u_r,x_0).
$$
By Proposition \ref{lemma 3} and Proposition \ref{help!!}, we have
$$
\theta(u\wedge v,x_0)=\theta(v,x_0),\quad 
\theta(u\vee v,x_0)=\theta(u,x_0).
$$
These equations, combined with the obvious relations
$$
(u\wedge v)(x_0)=v(x_0),\quad (u\vee v)(x_0)=u(x_0)
$$
lead to \eqref{tesi}.

\medskip\noindent
{\bf 2. Continuity.} Let $u_k$ be a sequence in $\fconvx$ epi-converging to $u\in\fconvx$. For simplicity, we 
set $w_k=(u_k)_r=\reg{u_k}{r}$ and $w=u_r\reg{u}{r}$.
\begin{eqnarray*}
\Big\vert\oZ_{\zeta,i}(u_k)-\oZ_{\zeta,i}(u)\Big\vert&=&\Big\vert\int_{\R^{2n}}\zeta(u_k(x),x,y)\d\Theta(u_k,(x,y))
-\int_{\R^{2n}}\zeta(u(x),x,y)\d\Theta(u,(x,y))\Big\vert\\
&=&\Big|\int_{\R^{2n}}\zeta(w_k(x),x,y)\d\Theta(u_k,(x,y))-\int_{\R^{2n}}\zeta(w(x),x,y)\d\Theta(u,(x,y))\Big|\\
&\le&\int_{\R^{2n}}\big|\zeta(w_k(x),x,y)-\zeta(w(x),x,y)\big|\d\Theta(u_k,(x,y))+\\
&+&\Big|\int_{\R^{2n}}\zeta(w(x),x,y)\d\Theta(u_k,(x,y))-\int_{\R^{2n}}\zeta(w(x),x,y)\d\Theta(u,(x,y))\Big|.
\end{eqnarray*}
By Proposition \ref{lemma 2}, the sequence $w_k$ epi-converges to $w$; as $w_k$ and $w$ are finite in $\R^n$, 
the convergence is uniform on compact sets. As $\zeta$ is continuous and with compact support with respect to $x$ and $y$, 
the sequence $\zeta(w_k(x),x,y)$ converges uniformly to $\zeta(w(x),x,y)$ in $\R^{2n}$ as $k\to+\infty$. Hence for an arbitrary $\varepsilon>0$, and 
for sufficiently large $k$, we have
$$
\int_{\R^{2n}}|\zeta(w_k(x),x,y)-\zeta(w(x),x,y)|\d\Theta(u_k,(x,y))\le\varepsilon\ \Theta(u_k,C)
$$
where $C$ is a compact subset of $\R^{2n}$ such that $\zeta(t,x,y)=0$ if $(x,y)\notin C$. On the other hand, as Hessian measures are locally finite,
there exists a constant $c>0$ depending only on $C$ such that
$$
\Theta(u_k,C)\le c
$$
for every $k\in\N$. This proves that
$$
\lim_{k\to+\infty}\int_{\R^{2n}}|\zeta(w_k(x),x,y)-\zeta(w(x),x,y)|\d\Theta(u_k,(x,y))=0.
$$

Set $\xi(x,y)=\zeta(w(x),x,y)$ for $(x,y)\in\R^n$. Recall that $w$ is a convex function such that $\dom(w)=\R^n$; in particular it is continuous
in $\R^n$, so that $\xi$ is continuous. Moreover, $\xi$ has compact support, as $\zeta$ has this property with respect to $(x,y)$. As the sequence of measures
$\Theta(u_k,\cdot)$ converges weakly to $\Theta(u,\cdot)$, we have
$$
\lim_{k\to+\infty}\int_{\R^{2n}}\zeta(w(x),x,y)\d\Theta(u_k,(x,y))-\int_{\R^{2n}}\zeta(w(x),x,y)\d\Theta(u,(x,y))=0,
$$
which proves the weak continuity.

\medskip
\noindent
{\bf 3. Simplicity.} Proposition \ref{theorem low dimension} implies simplicity.
\hfill$\square$\\

\subsection{Invariance properties}

\begin{proposition}\label{rotation invariance} Let $\zeta\in C(\R\times\R^n\times\R^n)$ have compact support with respect to the second and third variables.
If there exists $\xi\in C(\R\times[0,+\infty)^2)$ such that $
\zeta(t,x,y)=\xi(t,|x|,|y|)$ on $\R\times\R^n\times\R^n$, then
the valuation $\oZ_{\zeta, i}$ defined by \eqref{Hessian valuation - second type} is rotation invariant.
\end{proposition}

\begin{proof} Let $\phi\in\O(n)$, $u\in\fconvx$ and $u_\phi$ be defined as in Section \ref{Further properties of Hessian measures}. We have
\begin{eqnarray*}
\oZ_{\zeta, i}(u_\phi)&=&\int_{\Gamma_{u_\phi}}\xi(u_\phi(x), |x|,|y|)\d\Theta_i(u_\phi,(x,y)).
\end{eqnarray*}
As observed in the proof of Proposition \ref{behavior Hessian measures}, we have $\Gamma_{u_\phi}=\tilde \phi\,\Gamma_u$.
The conclusion follows by the change of variables
$$
(\bar x,\bar y)=(\phi x,\phi y)=\tilde \phi (x,y),
$$
and Proposition \ref{behavior Hessian measures}.
\end{proof}

\subsection{Composing a valuation with the conjugate function}
Let  $\G$ be a topological Abelian semigroup and $\oZ\colon\fconvx\to\G$ a continuous valuation. By Proposition \ref{minmaxconjugates} and 
Proposition~\ref{convergence conjugates}, the functional $\oZ^*\colon\fconvx\to\G$ defined by 
$$
\oZ^*(u)=\oZ(u^*)
$$ 
is a continuous valuation as well.

Assume that $\G=\R$ and that $\oZ=\oZ_{\zeta,i}$ is of the form \eqref{Hessian valuation - second type}. 
Then for every $u\in\fconvx$,
\begin{eqnarray*}
\oZ_{\zeta,i}^*(u)&=&\int_{\R^{2n}}\zeta(u^*(x),x,y)\d\Theta_i(u^*,(x,y))\\
&=&\int_{\R^{2n}}\zeta(\langle x,y\rangle-u(y),x,y)\d\Theta_i(u^*,(x,y))\\
&=&\int_{\R^{2n}}\zeta(\langle x,y\rangle-u(x),y,x)\d\Theta_{n-i}(u,(y,x))\\
&=&\oZ_{\bar\zeta,n-i}(u),
\end{eqnarray*}
where we have used (iii) of Lemma \ref{le:conjugate_subgradient} and Theorem \ref{duality}, and we set 
$$
\bar\zeta(t,x,y)=\zeta\big(\langle x,y\rangle-t,y,x\big)$$
on $\R\times\R^n\times\R^n$.
This shows in particular that $\oZ^*$ is still a Hessian valuation of the form \eqref{Hessian valuation - second type}.

\subsection{The cases $i=0$ and $i=n$}
Let $\oZ$ be of the form \eqref{Hessian valuation - second type}, with $i=n$:
$$
\oZ_{\zeta, n}(u)=\int_{\Gamma_u}\zeta(u(x),x,y)\d\Theta_n(u,(x,y)),
$$ 
for a suitable $\zeta\in C(\R\times\R^n\times\R^n)$, having compact support with respect to $x$ and $y$. 

Let $u\in\fconvx$; if $\dimension(\dom(u))<n$, then $\oZ_{\zeta,n}(u)=0$ by the simplicity property stated in Theorem 
\ref{theorem Hessian valuations}. Assume that $\dimension(\dom(u))=n$, and let
$$
D=\{x\in\interno(\dom(u))\colon\mbox{$u$ is differentiable at $x$}\}.
$$
As $u$ is differentiable $\hm^n$-a.e. on $\dom(u)$, we have
$$
\hm^n(\dom(u)\setminus D)=0.
$$
Hence, setting
$$
\Gamma_{u, D}=\{(x,y)\in\Gamma_u\colon x\in D\},
$$
by Proposition \ref{extreme cases}, 
\begin{eqnarray*}
\oZ_{\zeta,n}(u)&=&\int_{\Gamma_{u, D}}\zeta(x,u(x),\nabla u(x))\d\Theta_n(u,(x,y))\\
&=&\int_D \zeta(x,u(x),\nabla u(x))\d x.
\end{eqnarray*}
We conclude that 
$$
\oZ_{\zeta, n}(u)=\int_{\dom(u)} \zeta(x,u(x),\nabla u(x))\d x$$
for $u\in\fconvx$. Using this and Theorem \ref{duality} we also get 
$$
\oZ_{\zeta,0}(u)=\int_{\dom(u^*)}\zeta(\langle\nabla u^*(y),y\rangle-u^*(y),\nabla u^*(y),y)\d y$$
for $u\in\fconvx$.

\section{Valuations on  Convex and Coercive Functions}\label{Valuations on fconv}

We can now prove Theorem \ref{theorem intro 2} and first establish the following result.

\begin{theorem}\label{Hessian valuation - third type}
Let $\zeta\in C(\R\times\R^n)$ have compact support and let $i\in\{0,\dots, n\}$. The functional $\oZ_{\zeta,i}\colon\fconv\to\R$, defined by
\begin{equation}\label{third type}
\oZ_{\zeta, i}(u)=\int_{\R^{2n}}\zeta(u(x),y)\d\Theta_i(u,(x,y)),
\end{equation}
is a  continuous, translation invariant, $i$-simple valuation. Moreover, if there exists $\zeta\in C(\R\times[0,+\infty))$ such that
$\zeta(t,y)=\xi(t,|y|)$ on $\R\times\R^n$,
then $\oZ_{\zeta, i}$ is also rotation invariant.
\end{theorem}

\begin{proof} As $\zeta$ has compact support and  $u$  is coercive,  the function $(x,y) \mapsto \zeta(u(x),y)$ has compact support as well. 
By Lemma \ref{continuity lemma} this function is also continuous on $\Gamma_u$; hence, as Hessian measures are locally finite,
the integral in \eqref{third type} is always finite. The proofs of the valuation property, continuity and rotation invariance are the same as
in the case described in the previous section.  

In order to prove translation invariance, let $u\in\fconv$ and $x_0\in\R^n$. For the translated function $u_{x_0}$, we have
$$
\oZ_{\zeta, i}(u_{x_0})=\int_{\R^{2n}}\zeta(u(x-x_0),y)\d\Theta_i(u_{x_0},(x,y)). 
$$
The claim follows by the change of variables $(x-x_0,y)=(\bar x,\bar y)$ and Proposition \ref{behavior Hessian measures}. 
\end{proof}

\begin{remark}
If $\zeta$ is as in the previous theorem and $u\in\fconv\cap C^2(\R^n)$, then
$$
\oZ_{\zeta, i}(u)=\int_{\R^n}\zeta(u(x),\nabla u(x))\,[\Hess u(x)]_{n-i}\d x.
$$
Hence Theorem \ref{Hessian valuation - third type} implies Theorem \ref{theorem intro 2}.
\end{remark}

\begin{remark} We have seen that a sufficient condition on $\zeta\in C(\R\times\R^n)$ such that for $i\in\{0,\dots,n\}$, the integral
\begin{equation}\label{last section 1}
\oZ_{\zeta, i}(u)=\int_{\R^{2n}} \zeta(u(x),y)\d\Theta_i(u,(x,y))
\end{equation}
is finite for every $u\in\fconv$, is that $\zeta$ has compact support. On the other hand, in some special cases necessary and sufficient conditions for integrability 
are known. For instance, in \cite{Colesanti-Ludwig-Mussnig-1} is proved that if $i=n$ and $\zeta=\zeta(t)$ is non-negative, then \eqref{last section 1} is finite for every $u\in\fconv$ 
if and only if 
$$
\int_0^{+\infty} t^{n-1} \zeta(t)\d t<+\infty.
$$
It would be interesting to understand what the corresponding conditions are in the general case. In other words, it would be interesting to have for given $i$ an explicit description 
of the space of functions
$$
{\rm Int}_i(\R\times\R^n) =\{\zeta\in C(\R\times\R^n)\colon \mbox{$\displaystyle\int_{\R^{2n}} \zeta(u(x),y)\d\Theta_i(u,(x,y))<\infty$  for every $u\in\fconv$}\}.
$$ 
For the one-dimensional case, some progress in this direction was obtained in \cite{Paoli}.
\end{remark}

\goodbreak
\subsection{Comparison with level-set based valuations}\label{sub-section comparison}

Let $u\in\fconv$. For $t\in\R$, we consider the sublevel set $\{u\le t\}=\{x\in\R^n\colon u(x)\le t\}$.
By the semi-continuity, convexity and coercivity of $u$, this is either a compact convex subset of $\R^n$, that is, a convex body, or 
the empty set. In particular, for every $k\in\{0,\dots,n\}$, the $k$th quermassintegral $W_k(\{u\le t\})$ is well defined (with the convention
$W_k(\emptyset)=0$ for every $k$). 

Let $\omega\in C(\R)$ have compact support. For $k\in\{0,\dots,n\}$, functionals of the form
\begin{equation}\label{level set based valuations}
\oX_{k}(u)=\int_\R \omega(t)\,W_k(\{u\le t\})\d t
\end{equation}
have been considered on $\fconv$ in \cite{CavallinaColesanti} (with slightly different assumptions on $\omega$), and
in the context of log-concave or quasi-concave functions
in \cite{BobkovColesantiFragala,ColesantiLombardi,Milman-Rotem,Milman-Rotem2}. 

The functional in (\ref{level set based valuations}) defines a continuous, rigid motion invariant valuation on $\fconv$. Indeed, the assumption of compact support on $\omega$ and coercivity assure that it is finite. The valuation property was observed in \cite{CavallinaColesanti}, and 
rigid motion invariance is an immediate consequence of the rigid motion invariance of quermassintegrals. Continuity with respect to
epi-convergence can be proved using Lemma 5 in \cite{Colesanti-Ludwig-Mussnig-1}, which concerns the relation between epi-convergence 
and convergence of level sets. We refer to functionals of the form \eqref{level set based valuations} as to {\em level-set based valuations}. 

\medskip

When $k=n$, we have  $W_n(K)=\kappa_n$ (the volume of the unit ball in $\R^n$) for every (non-empty) convex body. 
Hence
$$
\oX_{n}(u)=\xi(\min\nolimits_{\R^n}u)
$$
for $u\in\fconv$, where, for $t\in\R$, we set
$$
\xi(t)=\kappa_n\int_t^\infty\omega(s)\d s.
$$
For $k=0$, as $W_0$ is the Lebesgue measure, by the Layer Cake Principle (see Section 6.3 in \cite{CavallinaColesanti}), 
we have the equivalent representation 
\begin{equation}\label{level set zero}
\oX_{0}(u)=\int_{\dom(u)} \xi(u(x))\d x
\end{equation}
for $u\in\fconv$, where, for $t\in\R$, we now set
$$
\xi(t)=\int_t^{+\infty}\omega(s)\d s.
$$
In particular, \eqref{level set zero} is a valuation which can be written as a Hessian valuation and as a level-set based valuation.

The following examples show that the family of Hessian valuations is not included in that of level-set based valuations.

\begin{example} Let $n=1$. Note that there are only two types of level-set based valuations on $\fconvone$:
$$
\oX_0(u)=\xi_0(\min\nolimits_{\R}u),\quad
\oX_1(u)=\int_{\dom(u)}\xi_1(u(x))\d x,
$$
where $\xi_0, \xi _1\in C(\R)$. Let $\zeta\in C(\R^2)$ be even with respect to second variable and have compact support. 
We consider the Hessian valuation 
$\oZ$ defined on $\fconv$ as
$$
\oZ(u)=\int_{\Gamma_u}\zeta(u(x),y)\d \Theta_0(u,(x,y)).
$$
Assume that $\oZ$ can be written as linear combination of level-set based valuations, that is,
\begin{equation}\label{linear combination}
\oZ(u)=\xi_0(\min(u))+\int_{\dom(u)}\xi_1(u(x))\d x\end{equation}
for suitable $\xi_0$ and $\xi_1$. Evaluating \eqref{linear combination} on functions of the form
$
u(x)=t+I_{[-r,r]}(x)
$
with $t\in\R$ and $r>0$ and using Proposition \ref{extreme cases}, we obtain 
$$
\int_{\R}\zeta(t,y)\d y=\xi_0(t)+2r \xi_1(t).
$$
As the last term is the only one depending on $r$, we get $\xi_1\equiv0$. Hence \eqref{linear combination} becomes
$$
\oZ(u)=\xi_0(\min\nolimits_{\R}u).
$$
Plugging the function
$
u(x)=t+s|x|
$
into the last equality with $t\in\R$ and $s>0$, we deduce
$$
\int_{-s}^s\zeta(t,y)\d y=
2\int_0^s \zeta(t,y)\d y=\xi_0(t)
$$
for $t\in\R$ and $s>0$. Thus we obtain $\zeta\equiv0$.
\end{example}

\begin{example} Let $\zeta=\zeta(t,s)$ be a function in $C(\R\times[0,\infty))$ of compact support. Consider the Hessian valuation $\oZ$ defined on $\fconv$ by
$$
\oZ(u)=\int_{\Gamma_u}\zeta(u(x),\vert y\vert)\d\Theta_n(u,(x,y))=
\int_{\dom(u)}\zeta(u(x),|\nabla u(x)|)\d x.$$
Assume that $\oZ$ can be written as linear combination of level-set based valuations
$$
\oZ(u)=\sum_{k=0}^n\oX_{k}(u).
$$
Note that $\oZ$ is $n$-simple,
{\em i.e.}, it vanishes for all $u$ such that $\dim(\dom(u))<n$. The only level-set based valuations with this property are 
those for $k=0$. It is not hard to see that this implies that the right hand-side of the previous inequality reduces
to the term $\oX_{0}$, hence, on $\fconv$,
\begin{equation}\label{linear combination 2}
\oZ(u)=\int_{\dom(u)} \xi(u(x))\d x
\end{equation}
for some function $\xi\in C(\R)$. We evaluate \eqref{linear combination 2} at functions $u$ of the form
$
u(x)=s|x|+I_{\Bn_r},
$
with $s, r>0$, obtaining 
$$
\int_0^r\zeta(s t,s)t^{n-1}\d t=\int_0^r\xi(s t) t^{n-1}\d t.
$$
As $r>0$ is arbitrary, we deduce
$
\zeta(t,s)=\xi(t)
$
for all  $t\in\R$ and $s>0$, which is possible only if $\xi$ does not depend on the second variable.
\end{example}


\section{The Space of 1-homogeneous Convex Functions}\label{The space of 1-homogeneous convex functions}

We analyze connections of the results on valuations on convex functions and valuations on convex bodies. Let
$$
\fconvh=\left\{u\in\fconvx\colon\dom(u)=\R^n,\ \mbox{$u$ is 1-homogeneous}\right\}.
$$
It is well known that $u\in\fconvh$ if and only if there exists a convex body $K$ in $\R^n$ such that $u=h_K$, the support function of $K$
(see Example \ref{example 2}). We recall that $\K^n$ denotes the family of convex bodies in $\R^n$.

We want to discuss valuations on $\fconvh$ and first look at Hessian measures on this space.
Denote by $u$ a generic element of $\fconvh$ and by $K$ the convex body such that $u=h_K$. For a point
$y\in\bd(K)$, let $N(K,y)$ be the normal cone to $K$ at $y$.
For $x\ne0$, we have $y\in\partial u(x)$ if and only if $y\in\bd(K)$ and $x\in N(K,y)$; moreover, $\partial u(0)=K$ (see \cite[Theorem 1.7.4]{Schneider}). 
Hence
$$
\Gamma_u=(\{0\}\times K)\cup\{(x,y)\in\R^{2n}\colon y\in\bd(K),\ x\in N(K,y)\}.
$$
In particular,
$$
\Gamma_u\subset\R^n\times K.
$$
The $1$-homogeneity of $u$ implies that the subdifferential is $0$-homogeneous, that is, $\partial u(t x)=\partial u(x)$ for $x\in\R^n$ and $t>0$.
Hence, for every $\eta\subset\R^{2n}$, $s\ge 0$ and $t>0$,
\begin{eqnarray*}
P_s(u,t\eta)&=&\{t x+s y\colon x\in\eta,\ y\in\partial u(x)\}\\
&=&t\left\{x+\frac s t y\colon(x,y)\in \eta\cap\Gamma_u\right\}\\
&=&t \, P_{s/t}(u,\eta).
\end{eqnarray*}
Taking Lebesgue measures and using the Steiner formula \eqref{Steiner}, we obtain the following homogeneity 
property for Hessian measures of support functions,
$$
\Theta_i(u,t\eta)=t^i\ \Theta_i(u,\eta)
$$
for 
$\eta\in\Borel(\R^{2n})$ and $i\in\{0,\dots,n\}$.

Hessian measures of support functions are closely connected to {\em support measures} of the corresponding convex bodies. For the definition
of the support measures $\Theta_0(K,\cdot), \dots, \Theta_n(K,\cdot)$ of a convex body $K$, we refer to \cite[Theorem 4.2.1]{Schneider}. The following relation 
was established in \cite[Corollary 5.9]{Colesanti-Hug-2000}. For every $\alpha\in\Borel(\sfe)$, $\beta\in\Borel(\R^n)$ and $i\in\{1,\dots,n\}$:
$$
\Theta_i(u,\hat\alpha\times\beta)=\tfrac 1n\ \Theta_{n-i}(K,\beta\times\alpha),
$$
where $\hat\alpha$  is the convex hull $\alpha\cup\{0\}$. 
In particular, 
$$
\Theta_i(u,\Bn\times\R^n)=\Theta_i(u,\Gamma_u\cap(\Bn\times\R^n))=\tfrac 1n\ \Theta_{n-i}(K,\R^n\times\sfe)=W_{n-i}(K)
$$
where $W_j(K)$, for $j\in\{0,\dots,n\}$, is the $j$th quermassintegral of $K$. The previous equality extends to the case $i=0$ as
\begin{eqnarray*}
\Theta_0(u,\Bn\times\R^n)=\Theta_0(u,\Gamma_u\cap(\Bn\times\R^n))=\hm^n(\pi_2(\Gamma_u))=\hm^n(K)=\ W_0(K).
\end{eqnarray*}

A real-valued valuation on $\K^n$ is a functional $\oY\colon\K^n\to\R$ which has the additivity property
$$
\oY(K\cup L)+\oY(K\cap L)=\oY(K)+\oY(L)
$$
for every $K,L\in\K^n$ such that $K\cup L\in\K^n$. The valuation $\oY$ is continuous, if it is continuous with respect to the Hausdorff metric.

\medskip\goodbreak
Hadwiger's celebrated classification theorem is the following result (see \cite{Hadwiger}). 

\begin{theorem}[Hadwiger] A functional $\oY\colon\K^n\to\R$ is a continuous valuation, invariant under rotations and translations if and only if
there exist $c_0,\dots,c_n\in\R$ such that
$$
\oY(K)=\sum_{i=0}^nc_i\ W_i(K)
$$ 
for every $K\in\K^n$. 
\end{theorem}

To rewrite Hadwiger's theorem as a result on $\fconvh$, we need the following facts. We associate with $\oY:\K^n\to \R$ the functional $\bar\oY: \fconvh\to \R$, defined as $\bar\oY(h_K)=\oY(K)$.
First, let $K_j$ be a sequence of convex bodies in $\R^n$ and let $u_j=h_{K_j}$ for every $j\in\N$. 
The sequence $K_j$ converges to a convex body $K$ with respect to the Hausdorff metric  if and only if the sequence of functions $u_j$ converges point\-wise to $u=h_K$ in $\R^n$
(and uniformly on compact sets). This is in turn equivalent to the fact that $u_j$ epi-converges to $u$. Hence, $\oY: \K^n \to \R$ is continuous with respect to the Hausdorff metric if and only if $\bar\oY: \fconvh \to \R$ is epi-continuous.

Second, if $K,L\in\K^n$ are such that $K\cup L\in\K^n$, then
$$
h_{K\cup L}=h_K\vee h_L,\quad
h_{K\cap L}=h_K\wedge h_L.
$$
This follows from Example \ref{example 2} and Proposition \ref{minmaxconjugates}.
Hence, if $\oY: \K^n \to \R$ is a valuation, then $\bar\oY: \fconvh \to \R$ is also a valuation.

Third, for every $\phi\in\O(n)$, the support function of $\phi K$ is given by
$h_{\phi K}(x)=h_K(\phi^t x)$ for every $x\in\R^n$. Hence, $\oY: \K^n \to \R$ is rotation invariant if and only if 
$\bar \oY$ is also rotation invariant.

Finally, for every $x_0\in\R^n$, the support function of the convex body $K+x_0$ is given by
$$
h_{K+x_0}(x)=h_K(x)+\langle x,x_0\rangle=u(x)+\langle x,x_0\rangle
$$
for every $x\in\R^n$. Hence, $\oY: \K^n \to \R$ is translation invariant if and only if $\bar\oY(u+w)= \bar\oY(w)$ for every linear 
function $w:\R^n\to \R$.  In general, we say that $\oZ:\fconvh \to \R$ is unchanged by addition of linear functions if 
$\oZ(u+w)=\oZ(u)$ for every $u\in\fconvh$ and linear function $w$.

As a consequence of the previous considerations, we obtain the following version of Hadwiger's theorem.

\begin{theorem}[Hadwiger] 
A functional $\oZ\colon\fconvh\to\R$ is a continuous and rotation invariant  valuation that is unchanged by addition of linear functions if and only if 
there exists $c_0,\dots,c_n\in\R$ such that
$$
\oZ(u)=\sum_{i=0}^n c_i\ \Theta_i(u,\Bn\times\R^n)
$$
for every $u\in\fconvh$.
\end{theorem}

\medskip
Without requiring that the functionals are unchanged by addition of linear functions, we obtain many further functionals.

\begin{theorem} Let $\zeta\in C(\R\times[0,+\infty))$. For $i\in\{0,\dots,n\}$, the functional $\oZ_{\zeta, i}\colon\fconvh\to\R$, defined by
\begin{equation}\label{Alesker type}
\oZ_{\zeta, i}(u)=\int_{\Bn\times\R^n} \zeta(u(x),|y|)\d\Theta_i(u,(x,y)),
\end{equation}
is a continuous and rotation invariant valuation on $\fconvh$. Consequently, $K\mapsto \oZ_{\zeta, i}(h_K)$ is 
a continuous and rotation invariant valuation on $\K^n$.
\end{theorem}

\begin{proof} Finiteness of $\oZ_{\zeta, i}$ follows from the compactness of $\Gamma_u\cap(\Bn\times\R^n)$, the continuity of $u$ and $\zeta$, and the local finiteness of Hessian measures.  
The proofs of the valuation property, continuity and rotation invariance are as in the proofs of Theorem \ref{theorem Hessian valuations}. 
\end{proof}

Alesker \cite{Alesker-99} obtained a classification of continuous and rotation invariant valuations on convex bodies that are polynomial with respect to translations and showed that these are functionals of type (\ref{Alesker type}). 

\subsection*{Acknowledgments}
The work of Monika Ludwig and Fabian Mussnig was supported, in part, by Austrian Science Fund (FWF) Project P25515-N25.  The work of Monika Ludwig was also supported by the National Science Foundation under Grant No. DMS-1440140 while she was in residence at the Mathematical Sciences Research Institute in Berkeley, California, during the Fall 2017 semester. The work of Andrea Colesanti 
was supported by the G.N.A.M.P.A. and by the F.I.R. project 2013: Geometrical and Qualitative Aspects of PDE's.



\end{document}